\documentclass[a4paper,11pt]{article}

\usepackage{amsfonts,amssymb,amsmath,amsthm,latexsym}
\usepackage{mathrsfs}
\usepackage{graphicx,graphics}
\usepackage{hyperref}
\usepackage{color}

\usepackage{caption}

\usepackage{subcaption}
\usepackage[english]{babel}
\makeatletter
\def\@captype{figure}
\makeatother

 \makeatletter
 \@addtoreset{equation}{section}
 \makeatother

 \makeatletter
 \@addtoreset{enunciato}{section}
 \makeatother

 \newcounter{enunciato}[section]

 \newtheorem{itassumption}{Assumption}
 \newtheorem{ittheorem}{Theorem}
 \newtheorem{itcorollary}{Corollary}
 \newtheorem{itlemma}{Lemma}
 \newtheorem{itproposition}{Proposition}
 \newtheorem{itdefinition}{Definition}
 \newtheorem{itremark}{Remark}
 \newtheorem{itclaim}{Claim}
 \newtheorem{itfact}{Fact}
 \newtheorem{itconjecture}{Conjecture}
 \newtheorem{itexample}{Example}

 \newenvironment{theorem}{\addtocounter{enunciato}{1}
 \begin{ittheorem}}{\end{ittheorem}}
 
 \newenvironment{corollary}{\addtocounter{enunciato}{1}
 \begin{itcorollary}}{\end{itcorollary}}

 \newenvironment{lemma}{\addtocounter{enunciato}{1}
 \begin{itlemma}}{\end{itlemma}}

 \newenvironment{proposition}{\addtocounter{enunciato}{1}
 \begin{itproposition}}{\end{itproposition}}

 \newenvironment{definition}{\addtocounter{enunciato}{1}
 \begin{itdefinition}}{\end{itdefinition}}

 \newenvironment{remark}{\addtocounter{enunciato}{1}
 \begin{itremark}}{\end{itremark}}

 \newenvironment{claim}{\addtocounter{enunciato}{1}
 \begin{itclaim}}{\end{itclaim}}

 \newenvironment{fact}{\addtocounter{enunciato}{1}
 \begin{itfact}}{\end{itfact}}

 \newenvironment{conjecture}{\addtocounter{enunciato}{1}
 \begin{itconjecture}}{\end{itconjecture}}
 
 \newenvironment{example}{\addtocounter{enunciato}{1}
 \begin{itexample}}{\end{itexample}}
 
 \newcommand{\be}[1]{\begin{equation}\label{#1}}
 \newcommand{\ee}{\end{equation}}

 \newcommand{\bl}[1]{\begin{lemma}\label{#1}}
 \newcommand{\el}{\end{lemma}}

 \newcommand{\br}[1]{\begin{remark}\label{#1}}
 \newcommand{\er}{\end{remark}}

 \newcommand{\bt}[1]{\begin{theorem}\label{#1}}
 \newcommand{\et}{\end{theorem}}

 \newcommand{\bd}[1]{\begin{definition}\label{#1}}
 \newcommand{\ed}{\end{definition}}

 \newcommand{\bcl}[1]{\begin{claim}\label{#1}}
 \newcommand{\ecl}{\end{claim}}

 \newcommand{\bfact}[1]{\begin{fact}\label{#1}}
 \newcommand{\efact}{\end{fact}}

 \newcommand{\bp}[1]{\begin{proposition}\label{#1}}
 \newcommand{\ep}{\end{proposition}}

 \newcommand{\bc}[1]{\begin{corollary}\label{#1}}
 \newcommand{\ec}{\end{corollary}}

 \newcommand{\bcj}[1]{\begin{conjecture}\label{#1}}
 \newcommand{\ecj}{\end{conjecture}}

 \newcommand{\bex}[1]{\begin{example}\label{#1}}
 \newcommand{\eex}{\end{example}}

 \newcommand{\bpr}{\begin{proof}}
 \newcommand{\epr}{\end{proof}}

 \newcommand{\bprt}[1]{\begin{proofoft}{\it\ref{#1}}.\,\,}
 \newcommand{\eprt}{\end{proofoft}}
 
 \newcommand{\bprc}[1]{\begin{proofofc}{\it\ref{#1}}.\,\,}
 \newcommand{\eprc}{\end{proofofc}}
 
 \newcommand{\bprp}[1]{\begin{proofofp}{\it\ref{#1}}.\,\,}
 \newcommand{\eprp}{\end{proofofp}}
 
 \newcommand{\bprl}[1]{\begin{proofofl}{\it\ref{#1}}.\,\,}
 \newcommand{\eprl}{\end{proofofl}}

 \newcommand{\bi}{\begin{itemize}}
 \newcommand{\ei}{\end{itemize}}

 \newcommand{\ben}{\begin{enumerate}}
 \newcommand{\een}{\end{enumerate}}
 
\newenvironment{proofoft}{\noindent 
{\bf Proof of Theorem\,}}{\hspace*{\fill}$\halmos$\medskip}
\newenvironment{proofofc}{\noindent 
{\bf Proof of Corollary\,}}{\hspace*{\fill}$\halmos$\medskip}
\newenvironment{proofofp}{\noindent 
{\bf Proof of Proposition\,}}{\hspace*{\fill}$\halmos$\medskip}
\newenvironment{proofofl}{\noindent 
{\bf Proof of Lemma\,}}{\hspace*{\fill}$\halmos$\medskip}
\newcommand{\halmos}{\rule{1ex}{1.4ex}}

\def \PP {{\mathbb P}}

\def \E {{\mathbb E}}
\def \N {{\mathbb N}}
\def \P {{\mathbb P}}

\def \R {{\mathbb R}}
\def \Z {{\mathbb Z}}

\def \cF {{\mathcal F}}

\newcommand{\one}{{\mathchoice {1\mskip-4mu\mathrm l}
         {1\mskip-4mu\mathrm l}
         {1\mskip-4.5mu\mathrm l}
         {1\mskip-5mu\mathrm l}}}

\newcommand{\ess}{ess\,sup}

\begin{document}

\title{The parabolic Anderson model in a dynamic\\ 
random environment: space-time ergodicity\\ 
for the quenched Lyapunov exponent}

\author{\renewcommand{\thefootnote}{\arabic{footnote}}
D.\ Erhard
\footnotemark[1]
\\
\renewcommand{\thefootnote}{\arabic{footnote}}
F.\ den Hollander
\footnotemark[2]
\\
\renewcommand{\thefootnote}{\arabic{footnote}}
G.\ Maillard
\footnotemark[3]
}

\footnotetext[1]{
Mathematical Institute, Leiden University, P.O.\ Box 9512,
2300 RA Leiden, The Netherlands,\\
{\sl erhardd@math.leidenuniv.nl}
}
\footnotetext[2]{
Mathematical Institute, Leiden University, P.O.\ Box 9512,
2300 RA Leiden, The Netherlands,\\
{\sl denholla@math.leidenuniv.nl}
}
\footnotetext[3]{
CMI-LATP, Aix-Marseille Universit\'e,
39 rue F. Joliot-Curie, F-13453 Marseille Cedex 13, France,\\
{\sl maillard@cmi.univ-mrs.fr}
}

\date{\today}
\maketitle


\begin{abstract}
We continue our study of the parabolic Anderson equation $\partial u(x,t)/\partial t 
= \kappa\Delta u(x,t) + \xi(x,t)u(x,t)$, $x\in\Z^d$, $t\geq 0$, where $\kappa \in 
[0,\infty)$ is the diffusion constant, $\Delta$ is the discrete Laplacian, and $\xi$
plays the role of a \emph{dynamic random environment} that drives the equation. The 
initial condition $u(x,0)=u_0(x)$, $x\in\Z^d$, is taken to be non-negative and bounded. 
The solution of the parabolic Anderson equation describes the evolution of a field of 
particles performing independent simple random walks with binary branching: particles 
jump at rate $2d\kappa$, split into two at rate $\xi \vee 0$, and die at rate $(-\xi) 
\vee 0$. 

We assume that $\xi$ is stationary and ergodic under translations in space and time, is 
not constant and satisfies $\E(|\xi(0,0)|)<\infty$, where $\E$ denotes expectation w.r.t.\ 
$\xi$. Our main object of interest is the \emph{quenched Lyapunov exponent} $\lambda_0
(\kappa) = \lim_{t\to\infty} \frac{1}{t}\log u(0,t)$. In earlier work \cite{GdHM11}, \cite{EdHM12} 
we established a number of basic properties of $\kappa\mapsto\lambda_0(\kappa)$ under 
certain mild space-time mixing and noisiness assumptions on $\xi$. In particular, we showed 
that the limit exists $\xi$-a.s., is finite and continuous on $[0,\infty)$, is globally Lipschitz on 
$(0,\infty)$, is not Lipschitz at $0$, and satisfies $\lambda_0(0) = \E(\xi(0,0))$ and 
$\lambda_0(\kappa) > \E(\xi(0,0))$ for $\kappa \in (0,\infty)$. 

In the present paper we show that $\lim_{\kappa\to\infty} \lambda_0(\kappa) =\E(\xi(0,0))$
under an additional space-time mixing condition on $\xi$ we call G\"artner-hyper-mixing. This 
result, which completes our study of the quenched Lyapunov exponent for general  $\xi$, shows 
that the parabolic Anderson model exhibits space-time ergodicity in the limit of large diffusivity. 
This fact is interesting because there are choices of $\xi$ that are G\"artner-hyper-mixing 
for which the \emph{annealed Lyapunov exponent} $\lambda_1(\kappa) = \lim_{t\to\infty} 
\frac{1}{t}\log \E(u(0,t))$ is infinite on $[0,\infty)$, a situation that is referred to as \emph{strongly 
catalytic behavior}. Our proof is based on a \emph{multiscale analysis} of $\xi$, in combination 
with discrete rearrangement inequalities for local times of simple random walk and spectral 
bounds for discrete Schr\"odinger operators.

{\it Acknowledgment.} DE and FdH were supported by ERC Advanced Grant 267356 VARIS. 
\end{abstract}

\newpage


\section{Introduction and main theorem}
\label{S1}

A fair amount is known about the behavior as a function of underlying parameters of 
the \emph{annealed Lyapunov exponents} for the parabolic Anderson model in a dynamic 
random environment. For an overview we refer the reader to \cite{GdHM08}. The main 
motivation behind the present paper is to understand the behavior of the \emph{quenched
Lyapunov exponent}, which is much harder to deal with. Our ultimate goal is to arrive
at a full qualitative picture of the quenched Lyapunov exponent for \emph{general} dynamic
random environments subject to certain mild space-time mixing and noisiness
assumptions.

Section~\ref{S1.1} defines the parabolic Anderson model and recalls the main results 
from \cite{GdHM11,EdHM12}. Section~\ref{S1.2} contains our main theorem, which states
that the quenched Lyapunov exponent converges to the average value of the environment 
in the limit of large diffusivity. Section~\ref{def} contains definitions, whereas Section~\ref{S1.4} discusses the main theorem, provides 
the necessary background, and gives a brief outline of the rest of the paper.


\subsection{Parabolic Anderson model}
\label{S1.1}

The parabolic Anderson model is the partial differential equation
\begin{equation}
\label{pA}
\frac{\partial}{\partial t}u(x,t) = \kappa\Delta u(x,t) + \xi(x,t)u(x,t),
\qquad x\in\Z^d,\,t\geq 0.
\end{equation}
Here, the $u$-field is $\R$-valued, $\kappa\in [0,\infty)$ is the diffusion
constant, $\Delta$ is the discrete Laplacian acting on $u$ as
\begin{equation}
\label{dL}
\Delta u(x,t) = \sum_{{y\in\Z^d} \atop {\|y-x\|=1}} [u(y,t)-u(x,t)]
\end{equation}
($\|\cdot\|$ is the $l_1$-norm), while
\begin{equation}
\label{rf}
\xi = (\xi_t)_{t \geq 0} \mbox{ with } \xi_t = \{\xi(x,t) \colon\,x\in\Z^d\}
\end{equation}
is an $\R$-valued random field playing the role of a \emph{dynamic random environment}
that drives the equation. As initial condition for (\ref{pA}) we take 
\begin{equation}
\label{ic}
u(x,0) = u_0(x),\,x\in\Z^d, \mbox{ with } u_0  \mbox{ non-negative, not identically zero, 
and bounded}.
\end{equation}

One interpretation of (\ref{pA}) and (\ref{ic}) comes from \emph{population dynamics}. 
Consider the special case where $\xi(x,t) = \gamma\xi_*(x,t)-\delta$ with
$\delta,\gamma \in (0,\infty)$ and $\xi_*$ an $\N_0$-valued random field. 
Consider a system of two types of particles, $A$ (catalyst) and $B$ (reactant), 
subject to:
\begin{itemize}
\item[--]
$A$-particles evolve autonomously according to a prescribed dynamics with 
$\xi_*(x,t)$ denoting the number of $A$-particles at site $x$ at time $t$;
\item[--]
$B$-particles perform independent simple random walks at rate $2d\kappa$ 
and split into two at a rate that is equal to $\gamma$ times the number of 
$A$-particles present at the same location at the same time;
\item[--]
$B$-particles die at rate $\delta$;
\item[--]
the average number of $B$-particles at site $x$ at time $0$ is $u_{0}(x)$.
\end{itemize}
Then
\begin{equation}
\label{uint}
\begin{array}{lll}
u(x,t) 
&=& \hbox{the average number of $B$-particles at site $x$ at time $t$}\\
&& \hbox{conditioned on the evolution of the $A$-particles}.
\end{array}
\end{equation}

The $\xi$-field is defined on a probability space $(\Omega,\cF,\PP)$. Throughout the 
paper we assume that 
\begin{equation}
\label{staterg}
\begin{aligned}
&\blacktriangleright\quad \xi \mbox{ is \emph{stationary} and \emph{ergodic} under 
translations in space and time.}\\
&\blacktriangleright\quad \xi \mbox{ is \emph{not constant} and } \E(|\xi(0,0)|)<\infty.
\end{aligned}
\end{equation}

The formal solution of \eqref{pA} is given by the \emph{Feynman-Kac formula}
\begin{equation}
\label{FK}
u(x,t) = E_x\left(\exp\left\{\int_0^t \xi(X^\kappa(s),t-s)\,ds\right\}\,
u_0(X^\kappa(t))\right),
\end{equation}
where $X^\kappa=(X^\kappa(t))_{t \geq 0}$ is the continuous-time simple random walk 
jumping at rate $2d\kappa$ (i.e., the Markov process with generator $\kappa\Delta$), 
and $P_x$ is the law of $X^\kappa$ when $X^\kappa(0)=x$. In \cite{EdHM12} we proved 
the following:
\begin{itemize}
\item[(0)]
Subject to the assumption  that $\xi$-a.s.\ $s \mapsto \xi(x,s)$ is locally integrable for every $x$
and that $\E(e^{q\xi(0,0)})<\infty$ for all $q \geq 0$, \eqref{FK} is finite for all $x,t$ and is the solution 
of \eqref{pA}.
\end{itemize}

The \emph{quenched Lyapunov exponent} associated with (\ref{pA}) is defined as
\begin{equation}
\label{Lyapexploc1}
\lambda_0(\kappa) = \lim_{t\to\infty} \frac{1}{t} \log u(0,t).
\end{equation}
In \cite{GdHM11} we showed that $\lambda_0(0) = \E(\xi(0,0))$ and $\lambda_0(\kappa) 
> \E(\xi(0,0))$ for $\kappa \in (0,\infty)$ as soon as the limit in \eqref{Lyapexploc1}
exists. In \cite{EdHM12} we proved the following:
\begin{itemize} 
\item[(1)]
Subject to certain \emph{space-time mixing assumptions} on $\xi$, the limit in
\eqref{Lyapexploc1} exists $\xi$-a.s.\ and in $L^{1}(\PP)$, is $\xi$-a.s.\ constant, 
is finite, and does not depend on $u_0$ satisfying \eqref{ic}.
\item[(2)]
Subject to certain additional \emph{noisiness assumptions} on $\xi$, $\kappa \mapsto 
\lambda_0(\kappa)$ is continuous on $[0,\infty)$, is globally Lipschitz on $(0,\infty)$, 
and is not Lipschitz at $0$.
\end{itemize}


\subsection{Main theorem and examples}
\label{S1.2}

Our main result is the following.

\bt{largekappa}
If $u_0=\delta_0$ and $\xi$ is G\"artner-hyper-mixing, then 
\begin{equation}
\label{eq:largelimit}
\lim_{\kappa \to\infty} \lambda_0(\kappa) = \E(\xi(0,0)).
\end{equation}
\et

The definition of G\"artner-hyper-mixing is given in 
Definitions~\ref{deltagood}--\ref{Gartnerposhypmix} below. A weaker form of these definitions 
was introduced and exploited in \cite{EdHM12}. Here are two examples of $\xi$-fields that 
are G\"artner-hyper-mixing.

\bex{mixingex} {\rm (See \cite{EdHM12})}\\
{\rm (e1)} 
Let $Y=(Y_t)_{t\geq 0}$ be a stationary and ergodic $\R$-valued Markov process satisfying
\begin{equation}
\label{Marksup}
\E\left[e^{q\sup_{t\in[0,1]} |Y_t|}\right] < \infty \qquad \forall\, q\geq 0.
\end{equation}
Let $(Y(x))_{x\in\Z^d}$ be a field of independent copies of $Y$. Then $\xi$ given by  $\xi(x,t)=Y_t(x)$ 
is G\"artner-hyper-mixing.\\
{\rm (e2)} 
Let $\xi$ be the zero-range process with rate function $g\colon\N_0\to (0,\infty)$ given by
$g(k) = k^{\beta}$, $\beta \in (0,1]$, and transition probabilities given by simple 
random walk on $\Z^d$. If $\xi$ starts from the product measure $\pi_{\rho}$, 
$\rho\in (0,\infty)$, with marginals
\begin{equation}
\label{ZRmeasure}
\forall\,x\in\Z^d\colon\,\quad \pi_{\rho}\big\{\eta \in \N_0^{\Z^d}\colon\, \eta(x)=k\big\} =
\begin{cases}
\gamma\,\frac{\rho^{k}}{\prod_{l=1}^kg(l)}, &\mbox{ if $k>0$}, \\
\gamma,  &\mbox{ if $k=0$},
\end{cases}
\end{equation}
where $\gamma \in (0,\infty)$ is a normalization constant, then $\xi$ is G\"artner-hyper-mixing.
\eex

\noindent
(The proof in \cite{EdHM12} is without the supremum in \eqref{BRgood} below, but easily carries 
over by inspection.) Example (e1) includes independent spin-flips, example (e2) includes 
independent random walks.

We expect that most interacting particle systems are G\"artner-hyper-mixing, including such 
classical systems as the stochastic Ising model, the contact process, the voter model and the 
exclusion process. Since these are bounded random fields, conditions (a2) and (a3) in 
Definition~\ref{Gartnerposhypmix} below are redundant and only condition (a1) needs to be 
verified. Note that the constant $\delta$ in \eqref{BRgood} below was allowed to be choosen arbitrarily large in \cite{EdHM12}. However, in this work we assume that $\delta$ goes to zero in a certain way (see Definition \eqref{Gartnerposhypmix}) so that (a1) indeed becomes an issue. We will not tackle the problem of solving this issue for the above mentioned fields in the present paper.   


\subsection{Definitions}
\label{def}

Throughout the rest of this paper we assume without loss of generality that $\E(\xi(0,0)) = 0$.

For $a_1,a_2,N \in \N$, denote by $\Delta_N(a_1,a_2)$ the set of $(\Z^d\times\N)$-valued 
sequences $\{(x_i,k_i)\}_{i=1}^N$ that are increasing with respect to the lexicographic 
ordering of $\Z^d\times\N$ and are such that, for all $1\leq i<j\leq N$, 
\begin{equation}
\label{mod} 
x_j\equiv x_i \, (\mathrm{mod}\,a_1), \qquad k_j\equiv k_i \, (\mathrm{mod}\,a_2).
\end{equation}
For $A \geq 1$, $\alpha>0$, $R\in\N$, $x\in\Z^d$ and $k,b,c\in\N_0$, define the \emph{space-time blocks}
(see Fig.~\ref{pictFiniteness2})
\begin{equation}
\label{Bblocks}
\tilde{B}_R^{A,\alpha}(x,k;b,c) = \left(\prod_{j=1}^{d}\big[(x(j)-1-b)\alpha A^R,(x(j)+1+b)\alpha A^R\big)
\cap\Z^d\right)\times[(k-c)A^R,(k+1)A^R).
\end{equation}
Abbreviate $B_R^{A,\alpha}(x,k)=\tilde{B}_R^{A,\alpha}(x,k;0,0)$
and $B_R^{A}(x,k)=B_R^{A,1}(x,k)$, and define the \emph{space-blocks}
\begin{equation}
\label{Qbox}
Q_R^{A,\alpha}(x) = x + [0,\alpha A^R)^d\cap\Z^d.
\end{equation}

\begin{figure}
\begin{center}
\setlength{\unitlength}{0.25cm}
\begin{picture}(28,23)(0,2.5)
\put(-10,-3){\vector(1,0){43}}
\put(-10,-3){\vector(0,1){33}}
\put(-1.5,21.5){\dashbox{0.5}(3,4)[]{}}
\put(22.5,21.5){\dashbox{0.5}(3,4)[]{}}
\put(-1.5,1.5){\dashbox{0.5}(3,4)[]{}}
\put(22.5,1.5){\dashbox{0.5}(3,4)[]{}}
\put(-4.5,20.5){\dashbox{0.5}(6,7)[]{}}
\put(19.5,20.5){\dashbox{0,5}(6,7)[]{}}
\put(-4.5,0.5){\dashbox{0.5}(6,7)[]{}}
\put(19.5,0.5){\dashbox{0.5}(6,7)[]{}}
\put(-2,21){\framebox(6,6)[]{}}
\put(22,21){\framebox(6,6)[]{}}
\put(-2,1){\framebox(6,6)[]{}}
\put(22,1){\framebox(6,6)[]{}}
\put(-6,18){\framebox(10,12)[]{}}
\put(18,18){\framebox(10,12)[]{}}
\put(-6,-2){\framebox(10,12)[]{}}
\put(18,-2){\framebox(10,12)[]{}}
\put(-0.9,-3.5){\vector(1,0){6}}
\put(-0.9,-3,5){\vector(-1,0){6}}
\put(-4.5,-5.5){$(c+1)A^{R+1}$}
\put(-6.5,-2.5){$\circledast_1$}
\put(-6.5,9.5){$\circledast_2$}
\put(3.5,9.5){$\circledast_3$}
\put(3.5,-2.5){$\circledast_4$}
\put(-5,0){$\circledast_5$}
\put(-2.5,0.5){$\circledast_6$}
\put(34,-4){time}
\put(-14.5,29){space}
\put(11,24){\vector(1,0){6}}
\put(11,24){\vector(-1,0){6}}
\put(6,25){$(a_2-c-1)A^{R+1}$}
\put(-1,14){\vector(0,1){3}}
\put(-1,14){\vector(0,-1){3}}
\put(0,14){$(a_1-2b-2)A^{R+1}$}
\end{picture}
\vspace{2.3cm}
\caption{\small The dashed blocks are $R$-blocks, i.e., $B_R^{A}(x,k)$ (inner) and 
$\tilde{B}_R^{A}(x,k;b,c)$ (outer) for some choice of $A,x,k,b,c$. The solid blocks 
are $(R+1)$-blocks, i.e., $B_{R+1}^{A}(y,l)$ (inner) and $\tilde{B}_{R+1}^{A}(y,l;b,c)$ 
(outer) for some choice of $A,y,l,b,c$ such that these $(R+1)$-blocks contain the 
corresponding $R$-blocks. All these blocks belong to the same equivalence class. The symbols $\{\circledast_i\}_{i=1,2,3,4,5,6}$ represents
 the space-time coordinates $\circledast_1 = ((y-1-b)A^{R+1},(l-c)A^{R+1})$, $\circledast_2
 =((y+1+b)A^{R+1},(l-c)A^{R+1})$, $\circledast_3= ((y+1+b)A^{R+1},(l+1)A^{R+1})$, 
$\circledast_4 = ((y-1-b)A^{R+1},(l+1)A^{R+1})$, $\circledast_5=((x-1-b)A^R, (k-c)A^R)$, 
$\circledast_6= ((y-1)A^{R+1},lA^{R+1})$.}
\label{pictFiniteness2}
\end{center}
\vspace{-.5cm}
\end{figure}
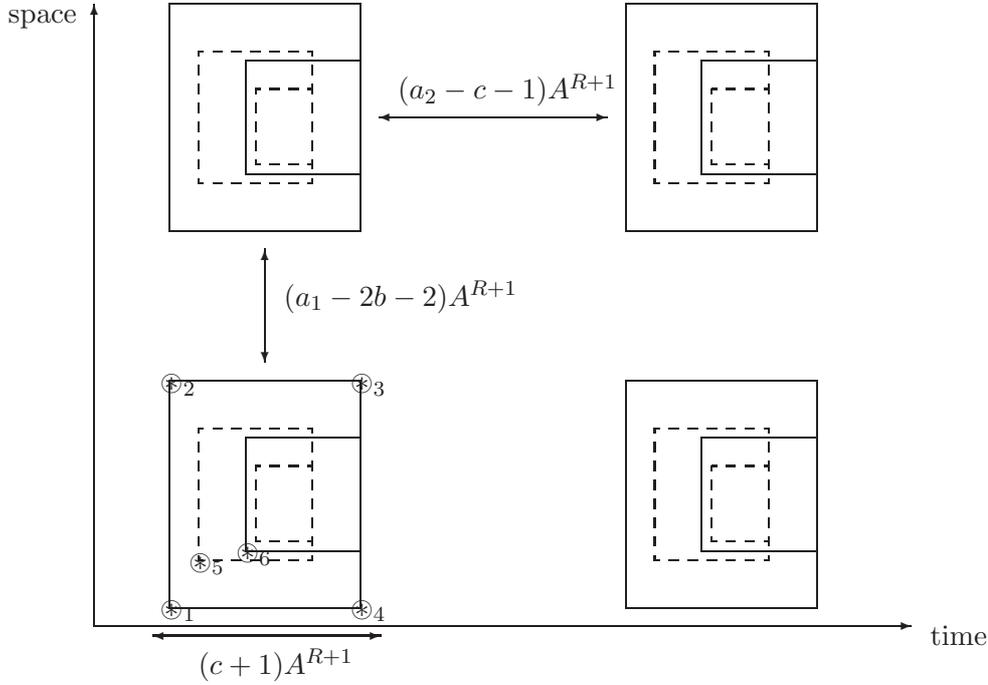

\bd{deltagood}{\rm {\bf [Good and bad blocks]}}\\
For $A \geq 1$, $\alpha>0$, $R \in \N$, $x \in \Z^d$, $m>0$, $k\in\N_0$, 
$\delta\in [0,\mathrm{\ess}\,[\xi(0,0)]]$ and $b,c\in\N_0$, the $R$-block $B_R^{A,\alpha}(x,k)$ is called $(\delta,b,c)$-good for the potential $\xi$ when, for all $s\in [(k-c)A^{R},(k+1)A^{R}-1/m)$,
\begin{equation}
\label{BRgood}
\frac{1}{|Q_R^{A,\alpha}(y)|}
\sum_{z \in Q_R^{A,\alpha}(y)} \sup_{r\in[s,s+1/m)}\xi(z,r) \leq \delta
\qquad \forall\,y\in\Z^d \colon\,
Q_R^{A,\alpha}(y)\times\{s\} \subseteq \tilde{B}_{R}^{A,\alpha}(x,k;b,c), 
\end{equation}
and is called $(\delta,b,c)$-bad otherwise.
\ed

For $A\geq 1$, $\alpha>0$, $R\in\N$, $x\in\Z^d$, $m >0$, $k\in\N_0$, $\delta\in
[0,\mathrm{\ess}\,[\xi(0,0)]]$ and $b,c\in\N_0$, let
\begin{equation}
\label{Aevent}
\begin{aligned}
&\mathcal{A}_R^{A,\alpha,\delta,m}(x,k;b,c)\\ 
&= \big\{\mbox{$B_{R+1}^{A,\alpha}(x,k)$ is $(\delta,b,c)$-good, 
but contains an $R$-block that is $(\delta,b,c)$-bad}\big\}.
\end{aligned}
\end{equation}

\bd{Gartnermix}{\rm {\bf [G\"artner-mixing]}}\\
The $\xi$-field is called $(A,\alpha,\delta,m,b,c)$-G\"artner-mixing when there are $a_1, a_2 \in \N$ 
such that
\begin{equation}
\label{wGartner}
\begin{aligned}
\sup_{(x_i,k_i)_{i=1}^N\in\Delta_N(a_1,a_2)}
\P\left(\bigcap_{i=1}^{N}\mathcal{A}_{R}^{A,\alpha, \delta,m}(x_i,k_i;b,c)\right)
\leq \big(A^{-4d(2d+1)(d+1)R}\big)^N \qquad \forall\,R\in\N,\,N\in\N.
\end{aligned}
\end{equation}
\ed

\bd{Gartnerposhypmix}{\rm {\bf [G\"artner-hyper-mixing]}}\\
The $\xi$-field is called G\"artner-hyper-mixing when the following conditions are satisfied:\\
{\rm (a1)}
There are $b,c\in\N_0$ and $K\geq 0$ such that for every $\delta >0$ there are 
$A_0>1$ and $m_0 >0$ such that $\xi\one\{\xi\geq K\}$ and $\xi$ are 
$(A,\alpha, \delta,m,b,c)$-G\"artner-mixing for all $A\geq A_0$, $m\geq m_0$ 
and all $\alpha\geq 1$, with $a_1,a_2$ in Definition~{\rm \ref{Gartnermix}} 
not depending on $A$, $m$ and $\alpha$.\\
{\rm (a2)} 
$\E[e^{q\sup_{s\in[0,1]} \xi(0,s)}] < \infty$ for all $q\geq 0$.\\
{\rm (a3)}
There are $R_0\in\N$ and $C_1\in[0,\mathrm{\ess}\,[\xi(0,0)]]$ such that 
\begin{equation}
\label{condexpgrowthB}
\P\left(\sup_{s \in [0,1]} \frac{1}{|B_R|}\sum_{y \in B_R} \xi(y,s) 
\geq C \right) \leq |B_R|^{-\alpha} \quad \forall\, R\geq R_0, C\geq C_1,
\end{equation}
for some $\alpha > [2d(2d+1)+1](d+2)/d$, where $B_R = [-R,R]^d\cap\Z^d$.
\ed


\subsection{Discussion}
\label{S1.4}

{\bf 1.}
What is interesting about Theorem~\ref{largekappa} is that it reveals a sharp 
contrast with what is known for the \emph{annealed Lyapunov exponent} 
\begin{equation}
\lambda_1(\kappa) = \lim_{t\to\infty} \frac{1}{t} \log \E(u(0,t)).
\end{equation} 
Indeed, there are choices of $\xi$ for which $\kappa \mapsto \lambda_1(\kappa)$ is 
everywhere infinite on $[0,\infty)$, a property referred to as \emph{strongly catalytic} 
behavior. For instance, as shown in \cite{GdH06}, if $\xi$ is $\gamma$ times a field 
of independent simple random walks starting in a Poisson equilibrium with arbitrary 
density, then this uniform divergence occurs in $d=1,2$ for $\gamma \in (0,\infty)$ 
and in $d \geq 3$ for $\gamma \in [1/G_d,\infty)$, with $G_d$ the Green function of 
simple random walk at the origin. By Example~\ref{mixingex}(e2) (with $\beta=1$),  this 
choice of $\xi$ is G\"artner-hyper-mixing.

\medskip\noindent
{\bf 2.} 
The annealed Lyapunov exponents 
\begin{equation}
\lambda_p(\kappa) = \lim_{t\to\infty} \frac{1}{t} \log \E([u(0,t)]^p), \qquad p \in \N,
\end{equation} 
were studied in detail in a series of papers where $\xi$ was chosen to evolve according to 
four specific interacting particle systems in equilibrium: independent Brownian motions, 
independent simple random walks, the simple symmetric exclusion process, and the voter 
model (for an overview, see \cite{GdHM08}). Their behavior turns out to be very different 
from that of $\lambda_0(\kappa)$. In \cite{EdHM12} it was conjectured that 
\begin{equation}
\lim_{\kappa\to\infty}[\lambda_p(\kappa)-\lambda_0(\kappa)]=0 \qquad \forall\,p\in\N
\end{equation}
because $\xi$ is ergodic in space and time. For the case where $\lambda_p(\kappa) \equiv 
\infty$ this statement is to be read as saying that $\lim_{\kappa\to\infty} \lambda_0
(\kappa)=\infty$. Theorem~\ref{largekappa} shows that this conjecture is false and that,
for $\xi$ G\"artner-hyper-mixing and satisfying conditions (0) and (2) in Section~\ref{S1.1}, 
the qualitative behavior of $\kappa\mapsto\lambda_0(\kappa)$ is as in Fig.~\ref{fig-lambda0}. 

\begin{figure}[htbp]
\vspace{0.7cm}
\begin{center}
\setlength{\unitlength}{0.25cm}
\begin{picture}(20,12)(1,1)
\put(0,2){\line(22,0){22}}
\put(0,2){\line(0,11){11}}
{\thicklines
\qbezier(0,2)(.4,8.8)(3,8)
\qbezier(3,8)(3.8,7.8)(5,6.6)
\qbezier(5,6.6)(9,2.6)(21,2.3)
}
\put(-1.3,1){$0$}
\put(0,2){\circle*{.35}}
\put(23,1.8){$\kappa$}
\put(-1,14){$\lambda_0(\kappa)$}
\end{picture}
\caption{\small Qualitative behavior of $\kappa\mapsto\lambda_0(\kappa)$.}
\label{fig-lambda0}
\end{center}
\vspace{-.3cm}
\end{figure}
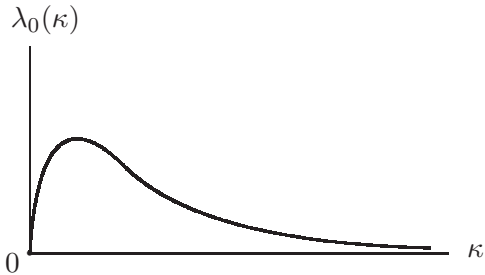

\medskip\noindent
{\bf 3.}
Our proof of Theorem~\ref{largekappa} is based on a \emph{multiscale analysis} of $\xi$,
in the spirit of \cite{KS03} and consists of two major steps: 
\begin{itemize}
\item[(I)]
We look at the bad $R$-blocks for all $R\in\N$. First we show that bad $R$-blocks are rare 
for large $R$. Next, using a \emph{discrete rearrangement inequality} for local times of simple 
random walk, we show that the contribution to the expectation in \eqref{FK} coming from bad 
$R$-blocks increases when we move them towards the origin. Therefore this contribution can 
be bounded from above by an expectation that pretends the bad $R$-blocks to be rearranged 
in a space-time cylinder around the origin. Since bad $R$-blocks are rare, this cylinder is narrow. 
Afterwards, because simple random walk is unlikely to spend a lot of time in narrow space-time 
cylinder, we are able to control the contribution coming from bad $R$-blocks to the expectation 
\eqref{FK} uniformly in $t$ and $\kappa$.
\item[(II)]
We look at the good $R$-blocks for all $R\in\N$. We control their contribution by using an 
\emph{eigenvalue expansion} of \eqref{FK}. An analysis of the largest eigenvalue in this 
expansion concludes the argument.     
\end{itemize}

\medskip
The remainder of this paper is organized as follows. In Section~\ref{S2} we formulate three 
key propositions and use these to prove Theorem~\ref{largekappa}. The three propositions 
are proved in Sections~\ref{S4}--\ref{S8}, respectively. In Appendix~\ref{Appendix} we prove 
two technical lemmas that are needed in Section~\ref{S5}, while in Appendix~\ref{Appendix*} 
we prove a spectral bound that is needed in Section~\ref{S8}.


\section{Three key propositions and proof of Theorem~\ref{largekappa}}
\label{S2}

To state our three key propositions we need some definitions. Fix $k_*\in\N$, 
and $t>0$. We say that $\Phi\colon\,[0,t]\to\Z^d$ is a path when
\begin{equation}
\label{path}
\|\Phi(s)-\Phi(s-)\| \leq 1 \qquad \forall\,s \in [0,t].
\end{equation}
Define the set of paths
\begin{equation}
\label{PiL}
\Pi(k_*,t,A) = \big\{\Phi\colon\,[0,t] \to \Z^d\colon\,\Phi\mbox{ crosses } 
k_* \mbox{ $1$-blocks}\big\}.
\end{equation} 

For $B_R\subseteq \Z^d \times [0,t]$, let $l_t(B_R)$ denote the local time of $X^{\kappa}$ in $B_R$ 
up to time $t$, and $l_t(\pi_1(B_R))$ the local time of $\pi_1(X^{\kappa})$ (the first coordinate of 
$X^{\kappa}$) in $\pi_1(B_R)$ up to time $t$. Furthermore, let $l_t(\mathrm{BAD}_{R}^{\delta}(\xi_K))$
and $l_t(\mathrm{BAD}^{\delta}(\xi))$ denote the local time of $X^\kappa$ in $(\delta,b,c)$-bad 
$R$-blocks up to time $t$ for the potential $\xi\one\{\xi\geq K\}$ and in $(\delta,b,c)$-bad $1$-blocks 
up to time $t$ for the potential $\xi$, respectively. Here and in the rest of the paper a bad $R$-block 
is $(\delta,b,c)$-bad for a choice of $K,\delta, b,c$ and some $A\geq A_0$, $m\geq m_0$, according 
to Definition~\ref{Gartnerposhypmix}.

In what follows, when we write sums like $\sum_{0 \leq k < t/A^R}$ or $\sum_{R=1}^{\varepsilon \log t}$ 
we will pretend that $t/A^R$ and $\varepsilon \log t$ are integer in order not to burden the notation with 
round off brackets. From the context it will always be clear where to place the brackets.

\bp{prop:FKest}
There is a $C_2>0$ such that for every $\varepsilon >0$ and $\delta>0$ there is an $A=A(\varepsilon)>3$, 
satisfying $\lim_{\varepsilon \downarrow 0} A(\varepsilon,\delta)=\infty$, such that $\xi$-a.s.\ for all 
$\kappa>0$ and all $t>0$ large enough,
\begin{equation}
\label{eq:FKinsuf}
\begin{aligned}
&E_0\Bigg(\exp\left\{\int_0^t\xi(X^{\kappa}(s),t-s)\, ds\right\}\Bigg) 
\leq e^{-t} + \Bigg[E_0\Bigg(\exp\left\{\int_0^{t}\overline{\xi}(X^{\kappa}(s),t-s)\, ds\right\}\Bigg)^{1/2}\\
&\qquad\qquad\times E_0\Bigg(\exp\left\{2\delta A^{d} l_t(\mathrm{BAD}^{\delta}(\xi))+ 
2\sum_{R=1}^{\varepsilon \log t} \delta A^{(R+1)d} l_t\big(\mathrm{BAD}_R^{\delta}(\xi_K)\big)\right\}\\
&\qquad\qquad\qquad\qquad\times
\one\big\{{\exists\, k_* \leq C_2\kappa t\colon\, X^{\kappa}\in\Pi(k_*,t,A)}\big\}\Bigg)^{1/2}\Bigg],
\end{aligned}
\end{equation}
where
\begin{equation}
\label{xibardef}
\overline{\xi}(x,s) = 2\xi(x,s)\one\big\{\xi(x,s)< \delta A^d, (x,s) 
\mbox{ is in a good $1$-block of }\xi\big\}.
\end{equation}
\ep

\bp{prop:boundedness}
There is a $C_2>0$ such that for every $\varepsilon, \tilde{\varepsilon}>0$ and $\delta>0$ there is an 
$A=A(\varepsilon,\tilde{\varepsilon},\delta)>3$, satisfying $\lim_{\tilde{\varepsilon} \downarrow 0} 
A(\varepsilon,\tilde{\varepsilon},\delta)=\infty$, such that $\xi$-a.s.\ for all $\kappa>0$ and all $t>0$ 
large enough,
\begin{equation}
\label{eq:controlbadblocks}
\begin{aligned}
&E_0\Bigg(\exp\left\{2\delta A^{d} l_t(\mathrm{BAD}^{\delta}(\xi))+ 
2\sum_{R=1}^{\varepsilon \log t} \delta A^{(R+1)d} l_t\big(\mathrm{BAD}_R^{\delta}(\xi_K)\big)\right\}\\
&\qquad\qquad\qquad\qquad\times
\one\big\{{\exists\, k_* \leq C_2\kappa t\colon\, X^{\kappa}\in\Pi(k_*,t,A)}\big\}\Bigg)
\leq e^{\tilde{\varepsilon}t}.
\end{aligned}
\end{equation}
\ep

\bp{prop:goodblocks}
There is a constant $C_3>0$ such that, for every $A>1$ and $\delta >0$,
\begin{equation}
\limsup_{\kappa\to\infty} \limsup_{n\to\infty} \frac{1}{An}
\log E_0\Bigg(\exp\left\{\int_0^{An}\overline{\xi}(X^{\kappa}(s),An-s)\, ds\right\}\Bigg)
\leq \frac{C_3}{A} +4\delta \qquad \xi-a.s.
\end{equation}
\ep

\noindent
Proposition~\ref{prop:FKest} estimates the Feynman-Kac formula in \eqref{FK} in terms 
of bad blocks and good blocks, Proposition~\ref{prop:boundedness} controls the contribution 
of bad block, while Proposition~\ref{prop:goodblocks} controls the contribution of good blocks.

We are now ready to prove Theorem~\ref{largekappa}.

\begin{proof}
Note that by Theorem 1.2(i) in \cite{GdHM11}, for all $\kappa\geq 0$ we have the 
lower bound $\lambda_0(\kappa)\geq 0$. Thus, it suffices to show the inequality in 
the reverse direction. To that end, fix $C_2, C_3>0$ according to 
Propositions~\ref{prop:FKest}--\ref{prop:goodblocks}, and fix $\varepsilon,\tilde{\varepsilon},
\delta>0$. According to Proposition~\ref{prop:boundedness}, there is an $A=A(\varepsilon,
\tilde{\varepsilon},\delta)$ such that, $\xi$-a.s.\ for all $\kappa>0$ and all $t$ of the form 
$t=An$ with $n\in\N$ large enough, the term in the left-hand side of (\ref{eq:controlbadblocks}) 
is bounded from above by $e^{\tilde{\varepsilon}An}$. According to Proposition~\ref{prop:goodblocks}, 
we have
\begin{equation}
\label{eq:applicationgoodblocks}
E_0\Bigg(\exp\left\{\int_0^{An}\overline{\xi}(X^{\kappa}(s),An-s)\, ds\right\}\Bigg)
\leq e^{4\delta An + C_3n + \chi(\kappa,n)}
\end{equation}
with $\limsup_{\kappa\to\infty} \limsup_{n\to\infty} \chi(\kappa,n)/n=0$. Proposition~\ref{prop:FKest}
therefore yields that, for all $\varepsilon,\tilde{\varepsilon},\delta>0$,
\begin{equation}
\label{eq:largekappa}
\limsup_{\kappa\to\infty} \lambda_0(\kappa) \leq \frac{C_3}{2A} + 2\delta +\frac{\tilde{\varepsilon}}{2}.
\end{equation}
Since $\lim_{\tilde{\varepsilon}\downarrow 0} A(\varepsilon,\tilde{\varepsilon},\delta) = \infty$ by 
Proposition \ref{prop:boundedness}, we get that for all $\delta>0$,
\begin{equation}
\label{eq:conclusion}
\limsup_{\kappa\to\infty} \lambda_0(\kappa) \leq 2\delta.
\end{equation}
Let $\delta \downarrow 0$ to get the claim.
\end{proof}


\section{Proof of Proposition~\ref{prop:FKest}}
\label{S4}

The proof is given in Section~\ref{S4.1} subject to Lemmas~\ref{lem:nobadblocks}--\ref{lem:nokblocks} 
below. The proof of these lemmas is given in Section~\ref{S4.2}.


\subsection{Proof of Proposition~\ref{prop:FKest} subject to two lemmas}
\label{S4.1}

For $A\geq 1$, $R\in\N$ and $\Phi \in \Pi(k_*,t,A)$, define
\begin{equation}
\begin{aligned}
\label{eq:Xi}
\Xi_R^{A}(\Phi)
&= \mbox{number of bad $R$-blocks crossed by $\Phi$},\\
\Xi_R^{A,k_*}
&= \sup_{\Phi \in \Pi(k_*,t,A)} \Xi_R^{A}(\Phi).
\end{aligned}
\end{equation}

\bl{lem:nobadblocks}
For every $\varepsilon>0$ there is an $A=A(\varepsilon)>3$ satisfying $\lim_{\varepsilon\downarrow 0}
A(\varepsilon)=\infty$ such that
\begin{equation}
\label{eq:nobadblocks}
\P\Big(\Xi_{R}^{A,k_*} > 0 \mbox{ for some } R\geq \varepsilon\log t \mbox{ and some }
k_*\in\N\Big)
\end{equation} 
is summable over $t\in\N$.  A possible choice is $A=e^{1/a\varepsilon[2d(2d+1)+1]}$ for some $a>1$.
\el

\bl{lem:nokblocks}
There is a $C_2>0$ such that $\xi$-a.s.\ for all $A>1$, all $t>0$ and all $\kappa>0$ large enough,
\begin{equation}
\label{eq:nokblocks}
E_0\Bigg(\exp\left\{ 
\int_0^t\xi(X^{\kappa}(s),t-s)\, ds\right\}
\one\big\{{\exists\, k_* > C_2\kappa t\colon\, X^{\kappa}\in\Pi(k_*,t,A)}\big\}\Bigg)
\leq e^{-t}.
\end{equation}
\el

We are now ready to prove Proposition \ref{prop:FKest}.

\bpr
Fix $C_2$ in accordance with Lemma \ref{lem:nokblocks} and $\varepsilon>0$. Let 
$\delta>0$ and fix $A>1$ according to Lemma \ref{lem:nobadblocks} such that 
$\delta A^d\geq K$, see Definition \ref{Gartnerposhypmix}. Note that
\begin{equation}
\label{FKsum}
\begin{aligned}
&E_0\left(\exp\left\{\int_0^t\xi(X^{\kappa}(s),t-s)\,ds\right\}
\one{\big\{\exists\, k_*\leq C_2\kappa t\colon\ X^{\kappa}\in \Pi(k_*,t,A)\big\}}\right)\\
&= E_0\Bigg(\exp\left\{\sum_{i=1}^{N(X^{\kappa},t)}
\int_{s_{i-1}}^{s_i}\xi(x_{i-1},t-u)\, du 
+ \int_{s_{N(X^{\kappa},t)}}^t \xi(x_{N(X^{\kappa},t)},t-u)\, du\right\}\\
&\qquad\qquad\qquad \times 
\one{\big\{\exists\, k_*\leq C_2\kappa t\colon\ X^{\kappa}\in \Pi(k_*,t,A)\big\}}\Bigg),
\end{aligned}
\end{equation}
where $N(X^{\kappa},t)$ is the number of jumps by $X^\kappa$ up to time $t$, 
$0=x_0,x_1,\dots,x_{N(X^{\kappa},t)}$ are the nearest-neighbor sites visited, and 
$0=s_0<s_1<\dots <s_{N(X^{\kappa},t)} \leq t$ are the jump times. To analyze (\ref{FKsum}), 
define
\begin{equation}
\label{eq:LambdaR}
\begin{aligned}
\Lambda_t(\mathrm{BAD}_R^{\delta}) 
&= \bigcup_{i=1}^{N(X^{\kappa},t)} \Big\{u \in [s_{i-1},s_{i})\colon\, 
\delta A^{Rd} < \xi(x_{i-1},t-u)\leq \delta A^{(R+1)d}\Big\} \\
&\qquad\qquad \bigcup \Big\{u\in [s_{N(X^{\kappa},t)},t)\colon\, 
\delta A^{Rd}< \xi(x_{N(X^{\kappa},t)},t-u)\leq \delta A^{(R+1)d}\Big\}.
\end{aligned}
\end{equation}
Then the contribution to the exponential in (\ref{FKsum}) may be bounded from above by
\begin{equation}
\label{contribution}
\int_0^t \xi(X^{\kappa}(s),t-s)\one\{\xi(X^{\kappa}(s),t-s) < \delta A^d\}\, ds 
+ \sum_{R\in\N} \delta A^{(R+1)d}\big|\Lambda_t\big(\mathrm{BAD}_R^{\delta}\big)\big|,
\end{equation}
By Definition~\ref{deltagood} and the fact that $\delta A^d\geq K$ (see the line preceding 
(\ref{FKsum})), if $\delta A^{Rd} < \xi(x_{i-1},t-u) \leq \delta A^{(R+1)d}$, then $(x_{i-1},t-u)$ 
belongs to a bad $R$-block for the potential $\xi\one\{\xi\geq K\}$. Hence
\begin{equation}
\label{LambdaRest}
\big|\Lambda_t\big(\mathrm{BAD}_R^{\delta}\big)\big| \leq l_t\big(\mathrm{BAD}_R^{\delta}(\xi_K)\big).
\end{equation}
To continue we write the indicator in (\ref{eq:LambdaR})
\begin{equation}
\label{eq:indicator}
\begin{aligned}
&\one\{\xi(X^{\kappa}(s),t-s)<\delta A^d, (X^{\kappa}(s),t-s) 
\mbox{ is in a good $1$-block of $\xi$}\}\\
&+\one\{\xi(X^{\kappa}(s),t-s)<\delta A^d, (X^{\kappa}(s),t-s) 
\mbox{ is in a bad $1$-block of $\xi$}\}
\end{aligned}
\end{equation}
By Lemma \ref{lem:nobadblocks} and our choice of $A$ at the beginning of the proof, $\xi$-a.s.\ for 
$t$ large enough there are no bad $R$-blocks with $R>\varepsilon \log t$. Thus, the expectation in 
the right-hand side of (\ref{FKsum}) may be estimated from above by
\begin{equation}
\label{eq:oldexpect}
\begin{aligned}
&E_0\Bigg(\exp\bigg\{\int_0^t\xi(X^{\kappa}(s),t-s)\\
&\quad\quad\times
\one\big\{\xi(X^{\kappa}(s),t-s)< \delta A^d,(X^{\kappa}(s),t-s)
\mbox{ is in a good $1$-block of $\xi$}\big\}\, ds\bigg\}\\
&\quad\quad\times
\exp\Bigg\{\delta A^dl_t(\mathrm{BAD}^{\delta}(\xi))+ \sum_{R=1}^{\varepsilon\log t}
\delta A^{(R+1)d}l_t(\mathrm{BAD}_R^{\delta}(\xi_K))\Bigg\}\\
&\quad\quad\times
\one\big\{\exists\ k_*\leq C_2\kappa t\colon\,
X^{\kappa}\in \Pi(k_*,t,A)\big\}\Bigg).
\end{aligned}
\end{equation}
Recall \eqref{xibardef}. An application of the Cauchy-Schwarz inequality yields the following upper bound 
for (\ref{eq:oldexpect}):
\begin{equation}
\label{eq:oldHolder}
\begin{aligned}
&E_0\Bigg(\exp\bigg\{\int_0^t\overline{\xi}(X^{\kappa}(s),t-s)
\, ds\bigg\}\Bigg)^{1/2}\\
&\times E_0\Bigg(\exp\Bigg\{2\delta A^dl_t(\mathrm{BAD}^{\delta}(\xi))+2\sum_{R=1}^{\varepsilon\log t}
\delta A^{(R+1)d}l_t(\mathrm{BAD}_R^{\delta}(\xi_K))\Bigg\}\\
&\qquad\qquad\qquad\qquad\qquad\qquad\times\one\big\{\exists\ k_*\leq C_2\kappa t\colon\,
X^{\kappa}\in \Pi(k_*,t,A)\big\}\Bigg)^{1/2}.
\end{aligned}
\end{equation}

The claim in \eqref{eq:FKinsuf} therefore follows by combining (\ref{FKsum}), 
(\ref{contribution}--\ref{LambdaRest}) and (\ref{eq:oldexpect})--(\ref{eq:oldHolder}) 
with Lemma~\ref{lem:nokblocks}.
\epr


\subsection{Proof of Lemmas~\ref{lem:nobadblocks}--\ref{lem:nokblocks} }
\label{S4.2}

\bpr
For the proof of Lemma~\ref{lem:nobadblocks}, see \cite[Lemma 3.3]{EdHM12}. To prove 
Lemma~\ref{lem:nokblocks}, use Cauchy-Schwarz to estimate the expectation in \eqref{eq:nokblocks}  
from  above by
\begin{equation}
\label{eq:nokblocksalt}
\Bigg[E_0\Bigg(\exp\left\{2\int_0^t\xi(X^{\kappa}(s),t-s)\, ds\right\}\Bigg)\Bigg]^{1/2}
\Bigg[P_0\Big(\exists\, k_* > C_2\kappa t\colon\, X^{\kappa}\in\Pi(k_*,t,A)\Big)\Bigg]^{1/2}.
\end{equation} 
To bound the first term in (\ref{eq:nokblocksalt}), note that by \cite[Eq.(3.54)]{EdHM12}  there is a 
$C>0$ such that $\xi$-a.s.\ for all $t,\kappa>0$,
\begin{equation}
\label{eq:oldresult}
E_0\Big(e^{2\int_0^t\xi(X^{\kappa}(s),t-s)\,ds}\Big) \leq e^{tC(\kappa +1)}.
\end{equation}
To bound the second term in (\ref{eq:nokblocksalt}) we use a similar strategy as for the
proof of Lemma~\ref{lem:crossing}.
Given $l_1, \ldots, l_{t/A}\in \N$, we say that $X^{\kappa}$ has label $(l_1,\ldots, l_{t/A})$
when $X^{\kappa}$ crosses $l_i$ $1$-blocks in the time interval $[(i-1)A,iA)$, $i\in\{1,\ldots, t/A\}$.
Fix $C_2>0$ and write 
\begin{equation}
\label{eq:partitionofkstar}
\begin{aligned}
&P_0\Bigg(X^{\kappa}\in\Pi(k_*,t,A) \mbox{ for some }k_*>C_2\kappa t\Bigg)\\
&= \sum_{j=1}^{\infty} P_0\Bigg(X^{\kappa}\in\Pi(k_*,t,A)
\mbox{ for some } k_*\in(jC_2\kappa t, (j+1)C_2\kappa t]\Bigg).
\end{aligned}
\end{equation}
For $j\in\N$, write $\sum_{(l_1^j,\ldots, l_{t/A}^j)}$ to denote the sum over all sequences $(l_1^j,\ldots, l_{t/A}^j)
\in\N^{t/A}$ with $jC_2\kappa t< \sum_{i=1}^{t/A} l_i^j \leq (j+1)C_2\kappa t$. Then each summand in (\ref{eq:partitionofkstar}) may, by an application of the Markov property, be rewritten as
\begin{equation}
\label{eq:labelsum}
\begin{aligned}
&\sum_{(l_1^j,\ldots, l_{t/A}^j)}E_0\Big(\one\{X^{\kappa} \mbox{ has label $(l_1^{j},\ldots, l_{t/A-1}^{j})$}\}
P_{X^{\kappa}(t-A)}\big(X^{\kappa} \mbox{ has label $l_{t/A}^{j}$}\big)\Big).
\end{aligned}
\end{equation}
Note that the number of jumps of a path $\Phi$ that visits $l_i^j$ $1$-blocks is at least $(l_i^j/2^d-1)A$. This 
is because for each $1$-block there are $(2^d-1)$ $1$-blocks with the same time coordinate at 
$l^{\infty}$-distance one. Hence, we may estimate (\ref{eq:labelsum}) from above by
\begin{equation}
\label{eq:jumpest}
\sum_{(l_1^j,\ldots, l_{t/A}^j)}E_0\Big(\one\{X^{\kappa} \mbox{ has label $(l_1^{j},\ldots, l_{t/A-1}^{j})$}\}\Big)
P_0\Big(N(X^{\kappa},A) \geq (l_i^j/2^d-1)A\Big),
\end{equation}
where $N(X^{\kappa},A)$ denotes the number of jumps of $X^{\kappa}$ in the time interval $[0,A)$.
An iteration of the arguments in (\ref{eq:labelsum}--\ref{eq:jumpest}), together with the tail estimate 
$P(\mathrm{POISSON}(\lambda)\geq k) \leq e^{-\lambda}(\lambda e)^{k}/k^k$, $k>2\lambda +1$, for 
Poisson random variables with mean $\lambda$, yields that for $C'>0$ large enough each summand in (\ref{eq:partitionofkstar}) is bounded from above by 
\begin{equation}
\label{eq:exitlblocks}
\begin{aligned}
&\sum_{(l_1^j,\ldots, l_{t/A}^j)}\prod_{l_i^j \geq \kappa C'}
P_0\Bigg(N(X^{\kappa},A) \geq (l_i^j/2^d-1)A\Bigg)\\
&\qquad \leq \sum_{(l_1^j,\ldots, l_{t/A}^j)}\prod_{l_i^j \geq \kappa C'} 
e^{-A2d\kappa}\exp\Big\{-(l_i^j/2^d-2)A \log([\kappa C'/2^d-2]/2d\kappa e)\Big\}.
\end{aligned}
\end{equation}
(It suffices to pick $C'$ such that $(C'/2^d-2)A\geq 4eAd\kappa +1$, which for $A>1$ and $\kappa >1$ 
is fulfilled when $C'\geq 2^d(4de +3)$.) Now note that if $C_2>2C'$, then for all $j\in\N$,
\begin{equation}
\label{eq:lowerboundonl}
\sum_{l_i^j \geq \kappa C'} l_i^j\geq \frac{jtC_2\kappa}{2}.
\end{equation}
Hence, inserting (\ref{eq:lowerboundonl}) into (\ref{eq:exitlblocks}), choosing $C_2$ large enough, and 
using the fact that for some $a,b\in (0,\infty)$ there are no more than $ae^{b\sqrt{C_2\kappa t}}$ such 
sequences $(l_1^j,  \ldots, l_{t/A}^j)$ (see \cite{HR18} or \cite{E42}), we get that for some $C''>0$ the 
left-hand side of (\ref{eq:partitionofkstar}) is bounded from above by $e^{-C''\kappa t}$. Inserting this 
bound into (\ref{eq:nokblocksalt}), using that $C''\to\infty$ as $C_2\to\infty$, and using (\ref{eq:oldresult}), 
we get the claim.
\epr


\section{Proof of Proposition \ref{prop:boundedness}}
\label{S5}

Proposition~\ref{prop:boundedness} is proved in Section~\ref{S5.2} subject to 
Propositions~\ref{prop:badintervals}--\ref{prop:localTime} below, which are stated 
in Section~\ref{S5.1} and proved in Sections~\ref{S6}--\ref{S7}.


\subsection{Two more propositions}
\label{S5.1}

Endow $\Z$ with the ordering $0\prec 1\prec -1\prec 2\prec -2\prec 3\prec  \cdots$. We say 
that two functions $f,g\colon\, \Z \to \R$ are equimeasurable when
\begin{equation}
\label{eq:niveauset}
|\{x\in\Z\colon\, f(x)> \lambda\}| = |\{x\in\Z\colon\, g(x)> \lambda\}| 
\qquad \forall\,\lambda \geq 0. 
\end{equation}
The \emph{symmetric decreasing rearrangement} of a function $f\colon\,\Z\to\R$ is defined to 
be the unique non-increasing function $f^{\sharp}\colon\,\Z \to \R$ that is equimeasurable with 
$f$. Given $A\subseteq \Z$, $A^{\sharp}\subseteq \Z$ is defined to be the unique set such that 
$(\one_A)^{\sharp} = \one_{A^{\sharp}}$. 

For $B\subseteq \Z^d\times [0,t]$, let $\pi_1(B)\subseteq\Z\times [0,t]$ be the projection of the spatial
part of $B$ onto its first spatial coordinate. Its one-dimensional symmetric decreasing rearrangement 
is the set
\begin{equation}
\label{eq:onedimrearrange}
\pi_1(B)^{\sharp} = \bigcup_{s\in [0,t]} 
\Big(\big\{x\in\Z\colon\,(x,s) \in \pi_1(B)\big\}^{\sharp} \times\{s\}\Big).
\end{equation}
For $A\geq 1$ and $R\in\N$, an \emph{$R$-interval} is a time-interval of the form $[kA^R, (k+1)A^R)$, 
$0 \leq k<t/A^{R}$. To make the proof more accessible, we no longer distinguish between badness 
referring to $\xi\one\{\xi\geq K\}$ and badness referring to $\xi$. Since both potentials satisfy the same 
mixing assumption {\rm(a1)} it will be clear from the proof that this does not affect the result.

\bp{prop:badintervals}
Let $\Phi\in \Pi(k_*,t,A)$. Then for all $A$ large enough there is a sequence $(\delta_R)_{R\in\N}$ 
in $(0,\infty)$ satisfying
\begin{equation}
\label{eq:deltaR}
\sum_{R\in\N} A^{Rd}\sqrt{\delta_R} <\infty
\end{equation}
such that $\xi$-a.s.\ the number of $R$-intervals in which $\Phi$ crosses more than $\delta_R k_*/(t/A)$ 
bad $R$-blocks is bounded from above by $\sqrt{\delta_R} t/A^{R}$. A possible choice is 
$\delta_R=K_1A^{-8d^2/3}A^{-4d(2d+1)R/3}$ for some $K_1>0$ not depending on $A$ and $R$.
\ep

\bp{prop:localTime}
For every $\varepsilon,t >0$, every sequence $(B_R)_{R\in\N}$ in $\Z^d\times [0,t]$ and every 
sequence $(C_R)_{R\in\N}$ in $[0,\infty)$ (see Fig.~{\rm \ref{pictRearrangement}}), 
\begin{equation}
\label{eq:symmest}
E_0\Bigg(\exp\Bigg\{\sum_{R=1}^{\varepsilon \log t}C_R l_t(B_{R})\Bigg\}\Bigg) \leq 
E_0\Bigg(\exp\Bigg\{\sum_{R=1}^{\varepsilon \log t}C_R l_t(\pi_1(B_{R})^{\sharp})\Bigg\}\Bigg).
\end{equation}
\ep

\begin{figure}[htbp]
\centering
\vspace{0.5cm}
\begin{subfigure}[b]{0.3\textwidth}
\centering
\setlength{\unitlength}{0.15cm}
\begin{picture}(28,23)(0,2.5)
\put(-13,9){\vector(1,0){38}}
\put(-13,-3){\vector(0,1){28}}
\put(-12.5,1.5){\framebox(3,3)}
\put(-12.5,18.5){\framebox(3,3)}
\put(-9.5,-2){\framebox(3,3)}
\put(-6.5,15){\framebox(3,3)}
\put(-6.5,22){\framebox(3,3)}
\put(-6.5,2){\framebox(3,6)}
\put(3.5,2.5){\framebox(3,3)}
\put(12.5,20){\framebox(3,6)}
\put(12.5,0){\framebox(3,6)}
\put(12.5,9){\framebox(3,3)}
\put(22,6){time}
\put(-20,23.5){space}
\end{picture}
\end{subfigure}
\hspace{1cm}
\begin{subfigure}[b]{0.3\textwidth}
\centering
\setlength{\unitlength}{0.15cm}
\begin{picture}(28,23)(0,2.5)
\put(0,9){\vector(1,0){38}}
\put(0,-3){\vector(0,1){28}}
\put(0.5,6){\framebox(3,6)}
\put(3.6,7.5){\framebox(2.8,3)}
\put(6.5,3){\framebox(3,12)}
\put(15.5,7.5){\framebox(3,3)}
\put(24.5,1.5){\framebox(3,15)}
\put(35,6){time}
\put(-7,23.5){space}
\end{picture}
\end{subfigure}
\vspace{1.5cm}
\caption{\small The picture on the left shows a configuration of space-time blocks before its rearrangement, 
the picture on the right after its rearrangement. Note that in each time-interval the total space volume of the 
blocks is the same in both configurations.}
\label{pictRearrangement}
\end{figure}
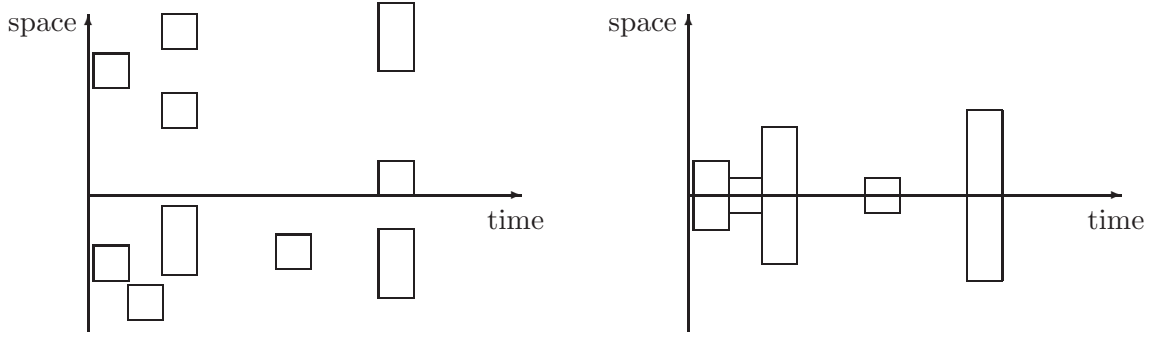

\subsection{Proof of Proposition \ref{prop:boundedness} subject to two propositions}
\label{S5.2}

\begin{proof}
Fix $\varepsilon >0$ and $A\geq 1$ according to Propositions~\ref{prop:FKest} and \ref{prop:badintervals}, 
and fix $\tilde{\varepsilon}>0$. The proof comes in 6 steps.

\medskip\noindent
{\bf 1.}
We begin by introducing some more notation. Define the space-time blocks
\begin{equation}
\label{eq:kappablock}
B_R^A(x,k;\kappa) = 
\Bigg(\prod_{j=1}^{d} [\sqrt{\kappa}(x(j)-1)A^{R},\sqrt{\kappa}(x(j)+1)A^{R})\cap\Z^d\Bigg)
\times [kA^{R},(k+1)A^{R}),
\end{equation}
which we call $(\kappa,R)$-blocks. These blocks are the same as $\tilde{B}_R^{A,\alpha}(x,k;0,0)=B_R^{A,\alpha}(x,k)$ 
in \eqref{Bblocks} with $\alpha=\sqrt{\kappa}$. Let 
$k_*(\kappa)$ denote the number of $(\kappa,1)$-blocks that are crossed by $X^{\kappa}$. For 
$k_*(\kappa)\in\N$  and $(x_i,k_i)_{0\leq i<k_*(\kappa)} \in (\Z^d\times \N)^{k_*(\kappa)}$,  write
\begin{equation}
\label{eq:fixedunion}
\bigcup_{(x_i,k_i)_{0\leq i<k_*(\kappa)}} B_1^{A}(x_i,k_i;\kappa)
\end{equation}
to denote the union over all the $(\kappa,1)$-blocks $B_1^{A}(x_i,k_i;\kappa)$, $0\leq i < k_*(\kappa)$.  
Likewise, write
\begin{equation}
\label{eq:sumoverblocks}
\sum_{(B_{1}^{A}(x_i,k_i;\kappa))_{0\leq i<k_*(\kappa)}}
\end{equation}
to denote the sum over all possible sequences of $(\kappa,1)$-blocks $B_1^{A}(x_i,k_i;\kappa)$, 
$0\leq i<k_*(\kappa)$ that can be crossed by a path $\Phi$. Finally, define
\begin{equation}
\label{eq:badselection}
\mathrm{BAD}_{R}^{\delta}((x_i,k_i)_{0\leq i< k_*(\kappa)})\\
= \bigg\{B_{R}^{A}(x,k)\colon\,B_{R}^{A}(x,k) \mbox{ is bad and intersects the union in 
(\ref{eq:fixedunion})}\bigg\}.
\end{equation}

\medskip\noindent
{\bf 2.}
We write $l_t(\mathrm{BAD}_R^{\delta})$ for the local time of $X^{\kappa}$ in $(\delta,b,c)$-bad 
$R$-blocks up to time $t$, where badness refers both to $\xi\one\{\xi\geq K\}$ and $\xi$. By 
(\ref{eq:controlbadblocks}), it is enough to show that for all $\kappa$ and $t$ large enough, 
\begin{equation}
\label{eq:tocontrol}
E_0\Bigg(\exp\left\{ 
4\sum_{R=1}^{\varepsilon \log t} \delta A^{(R+1)d} l_t\big(\mathrm{BAD}_R^{\delta}\big)\right\}
\one{\big\{\exists\, k_* \leq C_2\kappa t\colon\, X^{\kappa}\in\Pi(k_*,t,A)}\big\}\Bigg) 
\leq e^{\tilde{\varepsilon} t}.
\end{equation}
Recall (\ref{PiL}) to note that the left-hand side of (\ref{eq:tocontrol}) equals
\begin{equation}
\label{eq:sumoverk}
\sum_{k_*=t/A}^{C_2\kappa t}
E_0\Bigg(\exp\left\{ 
4\sum_{R=1}^{\varepsilon \log t} \delta A^{(R+1)d} l_t\big(\mathrm{BAD}_R^{\delta}\big)\right\}
\one\big\{{X^{\kappa} \mbox{ crosses $k_*$ $1$-blocks}}\big\}\Bigg).
\end{equation}
To prove \eqref{eq:tocontrol}, we attempt to apply Proposition \ref{prop:localTime}. To that end, for each $k_*$
we must sum over all configurations of $k_*$ $1$-blocks that may be crossed by $X^{\kappa}$. However, this 
sum is difficult to control, and therefore we do an additional \emph{coarse-graining of space-time} by considering 
$(\kappa,R)$-blocks instead of $R$-blocks. To that end we first note that there is a $C_4>0$ such that $k_*(\kappa) 
\leq C_4k_*/\sqrt{\kappa} +2t/A$ (see Lemma~\ref{lem:crossedblocks} in Section~\ref{S6.4} for a similar statement). 
To see why, note that if $\Phi$ crosses $l_i \leq \sqrt{\kappa}$ $1$-blocks in the time-interval $[(i-1)A,iA)$, 
$1\leq i \leq t/A$, then it crosses  $l_i^{\kappa} \leq 2$ $(\kappa,1)$-blocks in the same time-interval. Moreover, 
if $j\sqrt{\kappa}+1 \leq l_i \leq (j+1)\sqrt{\kappa}$ for some $j\in\N$, then $l_i^{\kappa} \leq j+2\leq (j+2)l_i/
j\sqrt{\kappa}$. Hence, the total number of $(\kappa,1)$-blocks that may be crossed by  $\Phi$ is bounded 
from above by
\begin{equation}
\label{eq:crossedkappablocks}
\sum_{i=1}^{t/A} l_i^{\kappa} \leq  \sum_{\substack{1 \leq i \leq t/A \\ l_i\leq \sqrt{\kappa}}} 2
+ \sum_{\substack{1 \leq i \leq t/A \\ l_i\geq\sqrt{\kappa}+1}} \frac{3l_i}{\sqrt{\kappa}} \leq 2t/A + \frac{3k_*}{\sqrt{\kappa}}.
\end{equation}
Thus, (\ref{eq:sumoverk}) is bounded from above by
\begin{equation}
\label{eq:sumoverkkappa}
\sum_{k_*(\kappa)=t/A}^{C_2C_4\sqrt{\kappa}t+2t/A}
E_0\Bigg(\exp\left\{ 
4\sum_{R=1}^{\varepsilon \log t} \delta A^{(R+1)d} l_t\big(\mathrm{BAD}_R^{\delta}\big)\right\}
\one\big\{{X^{\kappa} \mbox{ crosses $k_*(\kappa)$ $(\kappa,1)$-blocks}}\big\}\Bigg).
\end{equation}
To analyze (\ref{eq:sumoverkkappa}), fix $k_*(\kappa) \in [t/A,C_2C_4\sqrt{\kappa}t+2t/A]$. Summing over 
all possible ways to cross $k_*(\kappa)$ $(\kappa,1)$-blocks and recalling (\ref{eq:badselection}), we may 
estimate each summand in (\ref{eq:sumoverkkappa}) by
\begin{equation}
\label{eq:sumoverB}
\begin{aligned}
\sum_{(B_{1}^{A}(x_i,k_i;\kappa))_{0\leq i<k_*(\kappa)}}
E_0\Bigg(\exp&\left\{ 
4\sum_{R=1}^{\varepsilon \log t} \delta A^{(R+1)d} 
l_t\Big(\mathrm{BAD}_{R}^{\delta}\big((x_i,k_i)_{0\leq i< k_*(\kappa)}\big)\Big)\right\}\\
&\times\one\big\{{X^{\kappa} \mbox{ crosses }B_{1}^{A}(x_i,k_i;\kappa), 0\leq i< k_*(\kappa)}\big\}\Bigg).
\end{aligned}
\end{equation}
By Cauchy-Schwarz, (\ref{eq:sumoverB}) is at most
\begin{equation}
\label{eq:Holderbound}
\begin{aligned}
&\sum_{(B_{1}^{A}(x_i,k_i;\kappa))_{0\leq i<k_*(\kappa)}}
\Bigg[E_0\Bigg(\exp\left\{ 
8\sum_{R=1}^{\varepsilon \log t} \delta A^{(R+1)d} 
l_t\Big(\mathrm{BAD}_{R}^{\delta}\big((x_i,k_i)_{0\leq i< k_*(\kappa)}\big)
\Big)\right\}\Bigg)\Bigg]^{1/2}\\
&\qquad\qquad \times \left[P_0\Big(X^{\kappa} \mbox{ crosses } 
B_{1}^{A}(x_i,k_i,\kappa), 0\leq i< k_*(\kappa)\Big)\right]^{1/2}.
\end{aligned}
\end{equation}

\medskip\noindent
{\bf 3.}
By Proposition \ref{prop:localTime}, the first factor in the summand of (\ref{eq:Holderbound}) is 
not more than
\begin{equation}
\label{eq:appofrearrange}
\Bigg[E_0\Bigg(\exp\left\{ 
8\sum_{R=1}^{\varepsilon \log t} \delta A^{(R+1)d} l_t
\Big(\pi_1\big(\mathrm{BAD}_{R}^{\delta}((x_i,k_i)_{0\leq i< k_*(\kappa)})\big)^{\sharp}
\Big)\right\}\Bigg)\Bigg]^{1/2}.
\end{equation}
Next, if $X^{\kappa}$ crosses $k_*(\kappa)$ $(\kappa,1)$-blocks $B_{1}^{A}(x,k;\kappa)$, then a trivial 
counting estimate yields that $X^{\kappa}$ crosses at most $k_*(\kappa)\sqrt{\kappa}$ $1$-blocks.
Therefore, by Proposition \ref{prop:badintervals}, the number of $R$-intervals in which $X^{\kappa}$ 
crosses more than $\delta_R k_*(\kappa)\sqrt{\kappa}A/t$ bad $R$-blocks is bounded from above by 
$\sqrt{\delta_R} t/A^{R}$. We call these $R$-intervals $R$-atypical. Similarly, an $R$-interval is called
$R$-typical, if the number of bad $R$-blocks crossed by $X^{\kappa}$ is bounded by $\delta_R k_*(\kappa)
\sqrt{\kappa}A/t$. Define
\begin{equation}
\label{eq:Rstar}
R^{*}(k_*(\kappa)) = \max\big\{R\in\N\colon\, \delta_R k_*(\kappa)\sqrt{\kappa} A/t \geq 1\big\}.
\end{equation}
If $R>R^{*}(k_*(\kappa))$, then there are no bad $R$-blocks in $R$-typical intervals. (By the choice of 
$R$, their number is strictly less than one and therefore is zero.) Hence the local time in bad $R$-blocks 
is determined by the local time in bad $R$-blocks, which lie in $R$-atypical intervals. Consequently,
\begin{equation}
\label{eq:largeRlocal}
l_t\Big(\pi_1\big(\mathrm{BAD}_{R}^{\delta}((x_i,k_i)_{0\leq i<k_*(\kappa)})\big)^{\sharp}\Big)
\leq (\sqrt{\delta_R}t/A^R)A^R = \sqrt{\delta_R} t.
\end{equation}
On the other hand, if $1\leq R\leq R^*(k_*(\kappa))$ (see Fig.~\ref{pictAtypicalintervals}),  then there is 
a contribution coming from $R$-typical intervals as well, and so
\begin{equation}
\label{eq:symlocaltime}
l_t\Big(\pi_1\big(\mathrm{BAD}_{R}^{\delta}((x_i,k_i)_{0\leq i< k_*(\kappa)})\big)^{\sharp}\Big)
\leq \sqrt{\delta_R} t + l_t(\widetilde{B}_R(k_*(\kappa))),
\end{equation}
where
\begin{equation}
\label{eq:Btilde}
\widetilde{B}_R(k_*(\kappa)) = \bigg(\Big[-\tfrac{1}{2}A^{R}\delta_{R} k_*(\kappa)\sqrt{\kappa}A/t, 
\tfrac{1}{2}A^{R}\delta_{R} k_*(\kappa)\sqrt{\kappa}A/t\Big)\cap\Z\bigg) \times[0,t].
\end{equation}
Hence, (\ref{eq:appofrearrange}) is bounded from above by
\begin{equation}
\label{eq:upperbound}
\Bigg[E_0\Bigg(\exp\Bigg\{ 
8\sum_{R=1}^{R^*(k_*(\kappa))} \delta A^{(R+1)d} 
l_t\big(\widetilde{B}_R(k_*(\kappa))\big)\Bigg\}\Bigg)\Bigg]^{1/2}
\exp\Bigg\{4\sum_{R\in\N} \delta A^{(R+1)d}\sqrt{\delta_R}t\Bigg\}.
\end{equation}
For $A$ large enough, by Proposition~\ref{prop:badintervals} and the specific choice of $(\delta_R)_{R\in\N}$ in 
Propostion~\ref{prop:badintervals}, the sum in the second term is $\leq \tilde{\varepsilon}t/2$.

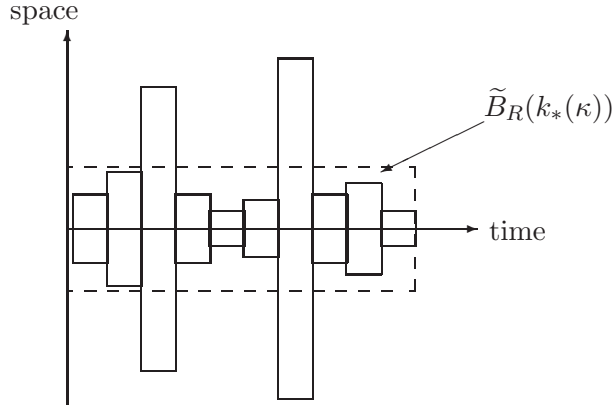
\begin{figure}
\centering
\vspace{0.5cm}
\begin{center}
\setlength{\unitlength}{0.15cm}
\begin{picture}(28,23)(0,2.5)
\put(-10,12.5){\vector(1,0){36}}
\put(-10,-3){\vector(0,1){33}}
\put(-9.5,9.5){\framebox(3,6)}
\put(-6.5,7.5){\framebox(3,10)}
\put(-3.5,0){\framebox(3,25)}
\put(-0.5,9.5){\framebox(3,6)}
\put(2.5,11){\framebox(3,3)}
\put(5.5,10){\framebox(3,5)}
\put(8.5,-2.5){\framebox(3,30)}
\put(11.5,9.5){\framebox(3,6)}
\put(14.5,8.5){\framebox(3,8)}
\put(17.5,11){\framebox(3,3)}
\put(-10,7){\dashbox(30.5,11)}
\put(27,11.5){time}
\put(-15,31){space}
\put(26.5,22){\vector(-2,-1){9}}
\put(26.5,22.5){$\widetilde{B}_R(k_*(\kappa))$}
\end{picture}
\vspace{1cm}
\caption{\small The picture shows a possible configuration of bad $R$-blocks after its rearrangement. 
There are two time-intervals in which the number of bad $R$-blocks is atypically large, i.e., larger than 
$\delta_R k_*(\kappa)\sqrt{\kappa}A/t$. The local time in these bad $R$-blocks can be bounded from 
above by the total length of these time-intervals, which is at most $\sqrt{\delta_R}t$. The local time of 
the bad $R$-blocks in the other time-intervals can be bounded from above by the local time of the 
enveloping dashed block, i.e., $\widetilde{B}_R(k_*(\kappa))$. }
\label{pictAtypicalintervals}
\end{center}
\end{figure}

\medskip\noindent
{\bf 4.} To estimate the first factor in \eqref{eq:upperbound} and control the second factor in the summand 
of \eqref{eq:Holderbound}, we need the following two lemmas whose proof is deferred to 
Appendix~\ref{Appendix}.

\begin{lemma}
\label{lem:LDPest}
Let $X^{\kappa}$ be simple random walk on $\Z$ with step rate $2\kappa$. There is a $K_2>0$ such that 
for all $\kappa>0$, all $n\in\N$, all $\beta_1,\beta_2,\ldots,\beta_n \geq 0$ and all nested finite intervals 
$\emptyset = I_0\subseteq I_1\subseteq I_2 \subseteq\cdots \subseteq I_n\subseteq\Z$,
\begin{equation}
\label{eq:localtimelem}
\log E_0\left(\exp\left\{\sum_{i=1}^{n}\beta_i \sum_{x \in I_i}
l_t(X^\kappa,x)\right\}\right)\\ 
\leq \frac{K_2t}{\sqrt{\kappa}}
\sum_{i=1}^{n}\Bigg[|I_i\setminus I_{i-1}|
\bigg(\sum_{j=i}^{n}\beta_j\bigg)^{3/2}\Bigg] +o(t),
\end{equation}
where $l_t(X^\kappa,x)$ is the local time of $X^\kappa$ at site $x$ up to time $t$.
\end{lemma}

\begin{lemma}
\label{lem:crossing}
There are $C_5,C_6>0$ such that for all $\kappa,t>0$ large enough, all $A> 0$ and all $k_*(\kappa) \geq C_5t$,
\begin{equation}
\label{eq:controlonkkappa}
P_0\Big(\mbox{$X^{\kappa}$ crosses $k_*(\kappa)$
$(\kappa,1)$-blocks}\Big) 
\leq e^{-C_6Ak_*(\kappa)}.
\end{equation}
\end{lemma}

Note that $A^{R+1}\delta_{R+1} < A^{R}\delta_R$, and so $\widetilde{B}_{R+1}(k_*(\kappa))\subseteq 
\widetilde{B}_{R}(k_*(\kappa))$ for all $R\in\N$. Moreover, for $k_*(\kappa)\leq C_2C_4\sqrt{\kappa}t+2t/A$ 
and  $1\leq R\leq R^*(k_*(\kappa))$ we have that the cardinality of the spatial part of the blocks defined in 
(\ref{eq:Btilde}) satisfies $|\widetilde{B}_{R} (k_*(\kappa))|\leq |\widetilde{B}_{1} (k_*(\kappa))| \leq A^2 
\delta_1C_2C_4\kappa + 2A\delta_1\sqrt{\kappa}$, which is \emph{bounded uniformly in $t$}. To apply 
Lemma~\ref{lem:LDPest}, we choose $t_0$ (which may depend on $\kappa$) such that for each family 
of intervals $I_1,\ldots, I_{R^*(k_*)}$, $k_*(\kappa) \in [t/A, C_2C_4\sqrt{\kappa} t +2t/A]$, with the property 
that $|I_i|\in[A, A^2\delta_1C_2C_4\kappa + 2A\delta_1\sqrt{\kappa}]$ for all $i\in\{1,\ldots, R^*(k_*(\kappa))\}$ 
the assertion of Lemma~\ref{lem:LDPest} holds uniformly in $t\geq t_0$. Then, for all $t\geq t_0$, the 
expectation in the left-hand side of (\ref{eq:upperbound}) is at most
\begin{equation}
\label{eq:estexponent}
\exp\Bigg\{\frac{K_2t}{\sqrt{\kappa}}\sum_{R=1}^{R^*(k_*(\kappa))}
\Bigg[\Big|\widetilde{B}_{R}(k_*(\kappa))\setminus \widetilde{B}_{R+1}(k_*(\kappa))\Big|
\bigg(\sum_{j=1}^{R}8\delta A^{(j+1)d}\bigg)^{3/2}\Bigg] + o(t)\Bigg\},
\end{equation}
where $\widetilde{B}_{R^*(k_*(\kappa))+1}(k_{*}(\kappa)) = \emptyset$. Next, note that
\begin{equation}
\label{eq:cardinalityofblocks}
|\widetilde{B}_{R}(k_*(\kappa))\setminus \widetilde{B}_{R+1}(k_*(\kappa))|
\leq \frac{A^{R}\delta_{R}k_*(\kappa)\sqrt{\kappa}A}{t}.
\end{equation}
Therefore the first term in the exponent of (\ref{eq:estexponent}), may be estimated from above by
\begin{equation}
\label{eq:estofexponent}
\begin{aligned}
\frac{K_2t}{\sqrt{\kappa}}\sum_{R=1}^{R^*(k_*(\kappa))}&
\Bigg[\frac{A^{R}\delta_{R}k_*(\kappa)\sqrt{\kappa}A}{t}
\bigg(\sum_{j=1}^{R}8\delta A^{(j+1)d}\bigg)^{3/2}\Bigg]\\
&\leq (8\delta)^{3/2}K_2 k_*(\kappa) A^{3d/2+1}\sum_{R=1}^{R^*(k_*(\kappa))}
\Bigg[A^R\delta_R\bigg(\sum_{j=1}^{R} A^{jd}\bigg)^{3/2}\Bigg].
\end{aligned}
\end{equation}
Furthermore,
\begin{equation}
\label{eq:geomsum}
\sum_{j=1}^{R} A^{jd} = \frac{A^d}{A^d-1}(A^{Rd}-1) \leq CA^{Rd},
\end{equation}
where $C>0$ does not depend on $A$. Hence, the right-hand side of (\ref{eq:estofexponent}) 
is at most
\begin{equation}
\label{eq:firstconclusion}
(8\delta )^{3/2}CK_2k_*(\kappa)A^{3d/2+1}\sum_{R=1}^{R^*(k_*(\kappa))}
A^R\delta_RA^{3Rd/2}.
\end{equation}
Recalling our choice of $\delta_R$ in Proposition~\ref{prop:badintervals}, we can estimate the sum in (\ref{eq:firstconclusion}) from above by
\begin{equation}
\label{eq:lastsum}
K_1A^{-8d^2/2}A^{-D(d)}
\bigg[\frac{1-A^{-R^{*}(k_*(\kappa))D(d)}}{1-A^{-D(d)}}\bigg].
\end{equation}
with $D(d)= (16d^2-d-6)/6>0$. Since $A>3$ by Proposition \ref{prop:FKest}, the last term in (\ref{eq:lastsum}) 
is bounded uniformly in $A$ and $R^*(k_*(\kappa))$. Inserting (\ref{eq:lastsum}) into (\ref{eq:firstconclusion}), 
we see that there is a $C_7>0$, not depending on $A$, such that the exponent in (\ref{eq:estexponent}) is 
bounded from above by $(8\delta)^{3/2}C_7A^{-D'(d)}k_*(\kappa) +o(t)$ with $D'(d)=(16d^2-5d-6)/3>0$, 
where $o(t)$ is uniform in $t\geq t_0$ for all $k_*(\kappa)\in[t/A, C_2C_4\sqrt{\kappa}t +2t/A]$.

\medskip\noindent
{\bf 5.}
It remains to estimate (recall  (\ref{eq:Holderbound}))
\begin{equation}
\label{eq:probaterm}
\sum_{(B_{1}^{A}(x_i,k_i;\kappa))_{0\leq i<k_*(\kappa)}}
\Bigg[P_0\Big(X^{\kappa} \mbox{ crosses }B_{1}^{A}(x_i,k_i,\kappa), 0\leq i< k_*(\kappa)\Big)\Bigg]^{1/2}.
\end{equation}
Let $|\sum_{(B_{1}^{A}(x_i,k_i;\kappa))_{0\leq i<k_*(\kappa)}}|$ denote the cardinality of the sum 
in  (\ref{eq:probaterm}). By Jensen's inequality, (\ref{eq:probaterm}) is not more than (recall 
(\ref{eq:sumoverblocks}))
\begin{equation}
\label{eq:Jensen}
\begin{aligned}
&\Bigg|\sum_{(B_{1}^{A}(x_i,k_i;\kappa))_{0\leq i<k_*(\kappa)}}\Bigg|^{1/2}
\Bigg[\sum_{(B_{1}^{A}(x_i,k_i;\kappa))_{0\leq i<k_*(\kappa)}}
P_0\Big(X^{\kappa} \mbox{ crosses }B_{1}^{A}(x_i,k_i;\kappa), 0\leq i< k_*(\kappa)\Big)\Bigg]^{1/2}\\
&= \Bigg|\sum_{(B_{1}^{A}(x_i,k_i;\kappa))_{0\leq i<k_*(\kappa)}}\Bigg|^{1/2}
\Bigg[P_0\Big(X^{\kappa} \mbox{ crosses $k_*(\kappa)$ $(\kappa,1)$-blocks}\Big)\Bigg]^{1/2}.
\end{aligned}
\end{equation}
To estimate the first term in the right-hand side of (\ref{eq:Jensen}), note that $|\sum_{(B_{1}^{A}(x_i,k_i;
\kappa))_{0\leq i<k_*(\kappa)}}|$ equals the number of different ways to visit $k_*(\kappa)$ $(\kappa,1)$-blocks.
Hence, there is a $C_{8}>0$ such that $|\sum_{(B_{1}^{A}(x_i,k_i;\kappa))_{0\leq i<k_*(\kappa)}}|$ is bounded 
from above by $e^{C_{8}k_*(\kappa)}$ (see also Lemma ~\ref{lem:cardinality} in Section~\ref{S6.4}).
Therefore, by Lemma~\ref{lem:crossing}, for $k_*(\kappa) \geq C_5t$ and $\kappa$ large enough, the 
right-hand side of (\ref{eq:Jensen}) may be estimated from above by
\begin{equation}
\label{eq:probabound}
e^{C_{8}k_*(\kappa)}\,e^{-C_6Ak_*(\kappa)}.
\end{equation}

\medskip\noindent
{\bf 6.}
We are now in a position to complete the proof of \eqref{eq:tocontrol}. Combining the estimates in 
\eqref{eq:appofrearrange}, \eqref{eq:upperbound} and (\ref{eq:estofexponent}--\ref{eq:probabound}),
we get for $t\geq t_0$ (see the lines following (\ref{eq:lastsum})),  
\begin{equation}
\label{eq:finalestimations}
\begin{aligned}
&E_0\Bigg(\exp\left\{ 
4\sum_{R=1}^{\varepsilon \log t} \delta A^{(R+1)d} l_t\big(\mathrm{BAD}_R\big)\right\}
\one{\big\{\exists\, k_* \leq C_2\kappa t\colon\, X^{\kappa}\in\Pi(k_*,t,A)}\big\}\Bigg)\\
&\qquad  \leq e^{\tilde{\varepsilon}t/2}\sum_{k_*(\kappa)=t/A}^{C_5 t-1} 
e^{(8\delta)^{3/2}C_{7}A^{-D'(d)}k_*(\kappa)+o(t)} \\
&\qquad\qquad\qquad + e^{\tilde{\varepsilon}t/2} \sum_{k_*(\kappa)=C_5 t}^{C_2C_4\sqrt{\kappa}t+2t/A}
e^{(8\delta)^{3/2}C_{7}A^{-D'(d)}k_*(\kappa)+o(t)}\,e^{C_{8}k_*(\kappa)} 
e^{-C_6Ak_*(\kappa)}\\
&\qquad \leq e^{\tilde{\varepsilon}t/2} C_5 t\, e^{(8\delta)^{3/2}C_{7}C_5 A^{-D'(d)}t+o(t)} 
+ e^{\tilde{\varepsilon}t/2}C_{9}\\
&\qquad\leq e^{2\tilde{\varepsilon}t},
\end{aligned}
\end{equation}
where we use that the sum in the third line of (\ref{eq:finalestimations}) is finite for $A$ large enough 
(which requires that $\varepsilon$ is small enough; recall Proposition~\ref{prop:FKest}). This settles 
\eqref{eq:tocontrol} and completes the proof of Proposition~\ref{prop:boundedness}. 
\end{proof}


\section{Proof of Proposition \ref{prop:badintervals}}
\label{S6}

The proof is given in Section~\ref{S6.1} subject to Lemma~\ref{lem:insuf} below. This lemma is stated 
in Section~\ref{S6.1} and proved in Sections~\ref{S6.2}--\ref{S6.5}. Recall the definition of $\Xi_R^{A,k_*}$ 
in (\ref{eq:Xi}). Throughout this section we abbreviate 
\begin{equation}
\widetilde{\delta}_R = A^{-2d(2d+1)R}.
\end{equation} 


\subsection{Proof of Proposition~\ref{prop:badintervals} subject to a further lemma}
\label{S6.1}

\bl{lem:insuf}
There is a $C>0$ such that $\xi$-a.s.\ for all $A$ and $m$ large enough, all $R\in\N$ and all $k_*\in\N$,
\begin{equation}
\label{eq:insuf}
\Xi_R^{A,k_*} \leq CA^{-(4d^2-1)}A^{-R}\widetilde{\delta}_Rk_*.
\end{equation}
\el

We are now ready to prove Proposition~\ref{prop:badintervals}.

\begin{proof}
Let $\Phi\in\Pi(k_*,t,A)$ and $R\in\N$. Suppose that there is a $\delta_R >0$ such that there are at least 
$\sqrt{\delta_R} t/A^{R}$ $R$-intervals in which $\Phi$ crosses more than $\delta_R k_*/(t/A)$ bad 
$R$-blocks. In all of these $R$-intervals $\Phi$ crosses at least 
\begin{equation}
\label{eq:badintervals}
\frac{\sqrt{\delta_R}t}{A^R}\,\frac{\delta_Rk_*}{(t/A)} = \delta_R^{3/2} A^{-(R-1)}k_*
\end{equation}
bad $R$-blocks. Lemma~\ref{lem:insuf} implies that $\xi$-a.s.\ $\delta_R^{3/2} A^{-(R-1)} \leq 
CA^{-(4d^2-1)}A^{-R}\widetilde{\delta}_R$, which is the same as $\delta_R^{3/2} \leq 
CA^{-4d^2}A^{-2d(2d+1)R}$. This yields the claim below \eqref{eq:deltaR} with $K_1=C^{2/3}$.
\end{proof}


\subsection{Proof of Lemma \ref{lem:insuf} subject to two further lemmas}
\label{S6.2}

The proof of Lemma~\ref{lem:insuf} is a modification of the proof of \cite[Lemma 3.5]{EdHM12} 
and is based on Lemmas~\ref{lem:multiscalebad}--\ref{lem:recursion} below, which count bad 
$R$-blocks. The proof of the second lemma is deferred to Section~\ref{S6.3}.
 
For $A\geq 1$, $R\in\N$ and $\Phi \in \Pi(k_*,t,A)$, define
\begin{equation}
\begin{aligned}
\Psi_R^A(\Phi) 
&= \mbox{number of good $(R+1)$-blocks crossed by $\Phi$ containing a bad $R$-block},\\
\Psi_R^{A,k_*} 
&= \sup_{\Phi \in \Pi(k_*,t,A)} \Psi_R^A(\Phi).
\end{aligned}
\end{equation}

\bl{lem:multiscalebad}
There is a $C'>0$ such that for all $A$ and $m$ large enough
\begin{equation}
\label{eq:multiscalepsi}
\P\Big(\Psi_{R}^{A,k_*} \geq C'A^{-R}\widetilde{\delta}_R k_*
\mbox{ for some } R\in\N \mbox{ and some } k_*\in\N_0\Big)
\end{equation}
is summable over $t\in\N$. A possible choice is $C'=3$. 
\el

\bl{lem:recursion}
For all $\varepsilon >0$ there is an $A=A(\varepsilon) >3$ such that $\xi$-a.s.\ there is a $t_0>0$ 
such that for all $R\in\N$, all $k_*\in\N$ and all $t\geq t_0$, 
\begin{equation}
\label{badblockest}
\Xi_{R}^{A,k_*}\leq A^{d+1}\sum_{i=1}^{\varepsilon\log t-R-1}2^{di}A^{(d+1)i}\Psi_{R+i}^{A,k_*}.
\end{equation}
\el

\begin{proof}
Lemma~\ref{lem:recursion} is the same as \cite[Lemma 3.7]{EdHM12}. The idea is to look at a bad 
$R$-block and check whether it is contained in a good $(R+1)$-block or in a bad $(R+1)$-block. An 
iteration over $R$, combined with a simple counting argument and Lemma~\ref{lem:nobadblocks}, 
yields the claim.
\end{proof}

We are now ready to prove Lemma~\ref{lem:insuf}.

\begin{proof}
By Lemma~\ref{lem:multiscalebad}, $\xi$-a.s.\ for $t$ large enough $\Psi_{R}^{A,k_*} \leq 
C'A^{-R}\widetilde{\delta}_R k_*$  for all $R\in\N$ and all $k_*\in\N$. By Lemma~\ref{lem:recursion}, 
recalling that $\widetilde{\delta}_R = A^{-2d(2d+1)R}$, we may estimate
\begin{equation}
\label{eq:propest}
\begin{aligned}
\Xi_{R}^{A,k_*} 
&\leq A^{d+1}\sum_{i\geq 1} 2^{di}A^{(d+1)i}C'A^{-(R+i)}\widetilde{\delta}_{R+i} k_*\\
&= C'A^{d+1}A^{-R}\widetilde{\delta}_Rk_* \sum_{i\geq 1}2^{di} A^{(d+1)i}A^{-i}A^{-2d(2d+1)i}\\
&= C'A^{d+1}A^{-R}\widetilde{\delta}_Rk_*  \frac{2^dA^dA^{-2d(2d+1)}}{1-2^dA^{d}A^{-2d(2d+1)}}.
\end{aligned}
\end{equation}
Note that for $A \geq A_0>1$ there is a $C>0$, depending on $A_0$ but not on $A$, such that 
the term in the right-hand side of (\ref{eq:propest}) is bounded from above by
\begin{equation}
\label{eq:upperest}
CA^{-(4d^2-1)}A^{-R}\widetilde{\delta}_Rk_*,
\end{equation}
which yields the claim.
\end{proof}


\subsection{Proof of Lemma \ref{lem:multiscalebad} subject to a further lemma}
\label{S6.3}

The proof of Lemma~\ref{lem:multiscalebad} is based on Lemma~\ref{insufblocks} below.
Let $x\in \Z^d$ and $k,R \in \N$. Abbreviate
\begin{equation}
\label{eq:chi}
\chi^{A}(x,k) = \one\big\{B_{R+1}^{A}(x,k) 
\mbox{ is good but contains a bad $R$-block}\big\}.
\end{equation}

\bl{insufblocks}
There is  a $C>0$ such that for all $A$ and $m$ large enough, all $R\in \N$ and all $k_*\in\N$,
\begin{equation}
\label{eq:insblocks}
\P\Big(
\begin{array}{ll}
&\mbox{there is a path that crosses $k_*$ $1$-blocks and intersects }\\ 
&\mbox{at least $3A^{-R}\widetilde{\delta}_Rk_*$ blocks $B_{R+1}^{A}(x,k)$ with $\chi^A(x,k)=1$}
\end{array}
\Big)
\leq \exp\big\{-CA^{-R}\widetilde{\delta}_R k_*\big\}.
\end{equation} 
\el

We are now ready to prove Lemma~\ref{lem:multiscalebad}.

\begin{proof}
First note that $k_*\geq t/A$ and that, $\xi$-a.s.\ for $t$ large enough, $1\leq R \leq \varepsilon \log t$, 
by Lemma \ref{lem:nobadblocks}. For each such $R$, we have by Lemma~\ref{insufblocks},
\begin{equation}
\label{eq:multiest}
\begin{aligned}
&\P\Big(
\begin{array}{ll}
&\mbox{there is a path that crosses $k_*$ $1$-blocks and intersects at least}\\ 
&\mbox{$3A^{-R}\widetilde{\delta}_Rk_*$ blocks $B_{R+1}^{A}(x,k)$ with $\chi^{A}(x,k)=1$ 
for some  $k_*\geq t/A$} 
\end{array}
\Big)\\
&\qquad \leq \sum_{k_*\geq t/A}\exp\{-CA^{-R}\widetilde{\delta}_R k_*\}
\leq \frac{\exp\{-CA^{-R}\widetilde{\delta}_R t/A\}}
{1-\exp\{-CA^{-R}\widetilde{\delta}_R\}}.
\end{aligned}
\end{equation}
Because $1\leq R\leq \varepsilon \log t$ and $R \mapsto  A^{-R}\widetilde{\delta}_R$ is non-increasing, 
the numerator in the right-hand side of \eqref{eq:multiest}  is bounded from above by 
$\exp\{-CA^{-\varepsilon \log t}\widetilde{\delta}_{\varepsilon \log t}t/A\}$ while the denominator is 
bounded from below by $1-\exp\{-CA^{-\varepsilon \log t}\widetilde{\delta}_{\varepsilon \log t}\}$. Using 
the choice of $A$ in Lemma~\ref{lem:nobadblocks},  we see that (\ref{eq:multiest}) 
is bounded from above by
\begin{equation}
\frac{\exp\{-Ct^{1-a^{-1}}/A\}}
{1-\exp\{-Ct^{-a^{-1}}\}}, \qquad a>1. 
\end{equation}
Note that this is of order $\exp\{-C't^{\tilde{\varepsilon}}\}$ for some $C',\tilde{\varepsilon}>0$, and so the 
probability in (\ref{eq:multiscalepsi}) is bounded from above by $(\varepsilon \log t) 
\exp\{-C't^{\tilde{\varepsilon}}\}$, which is summable over $t\in\N$.
\end{proof}


\subsection{Proof of Lemma \ref{insufblocks} subject to two further lemmas}
\label{S6.4}

The proof of Lemma \ref{insufblocks} is based on Lemmas~\ref{lem:cardinality}--\ref{lem:crossedblocks} 
below, which are proved in Section~\ref{S6.5}.

\bpr 
Our first further lemma reads:

\bl{lem:cardinality}
There is a $C>0$ such that for all $l\in\N$ and $R\in\N$ there are no more than
$e^{Cl}$ possible ways for $\Phi$ to visit at most $l$ $R$-blocks.
\el

Fix $R\in\N$. We divide blocks into equivalence classes such that blocks belonging to the same 
equivalence class can essentially be treated as independent. To that end, we take $a_1, a_2 \in\N$ 
according to condition {\rm(a1)} in Definition \ref{Gartnerposhypmix} and say that $(x,k)$ and $(x',k')$ 
are equivalent when
\begin{equation}
\label{eq:mod}
x = x' \, (\mathrm{mod}\,a_1), \qquad k = k' \, (\mathrm{mod}\,a_2).
\end{equation}
We denote the set of corresponding representants by $([x],[k])$, and write $\sum_{([x],[k])}$ 
to denote the sum over all equivalence classes. Note that the 
left-hand side of (\ref{eq:insblocks}) is bounded from above by
\begin{equation}
\label{eq:sumbound}
\sum_{([x],[k])} 
\P\Bigg(\begin{array}{ll}
&\mbox{there is a path that crosses $k_*$ $1$-blocks and intersects}\\
&\mbox{at least $3A^{-R}\widetilde{\delta}_R k_*/a_1^{d}a_2$ 
blocks $B_{R+1}^{A}(x,k)$}\\
&\mbox{with $\chi^{A}(x,k)=1$ and $(x,k)\equiv ([x],[k])$}
\end{array}
\Bigg).
\end{equation}
Fix an equivalence class. Put $\rho_R = A^{-4d(2d+1)(d+1)R}$ (recall (\ref{wGartner})). To control the 
cardinality of the number of different ways to visit a given number of $(R+1)$-blocks, we consider 
enlarged blocks, namely, we let
\begin{equation}
L=L(R) = (1/\rho_R)^{1/(d+1)}
\end{equation} 
and define
\begin{equation}
\label{eq:enblock}
\tilde{B}_{R}^{A}(x,k) 
= \left(\prod_{j=1}^{d}
\big[L x(j)A^{R},L(x(j)+1)A^{R}\big)\cap\Z^d\right)\times[L kA^{R},L(k+1)A^{R}).
\end{equation}

Our second further lemma reads:

\bl{lem:crossedblocks}
If $\Phi$ crosses $k_*$ $1$-blocks, then for all $R\in\N$ it crosses no more than $l_R=3k_*/A^{R-1}L$ 
blocks $\widetilde{B}_R^{A}(x,k)$.
\el
 
We write 
\begin{equation}
\label{eq:allseq}
\bigcup_{(x_i,k_i)_{0\leq i<l_{R+1}}} \tilde{B}_{R+1}^{A}(x_i,k_i)
\end{equation}
to denote the union over at most $l_{R+1}$ blocks $\widetilde{B}_R^{A}(x,k)$ and 
\begin{equation}
\label{eq:allsum}
\sum_{(\tilde{B}_{R+1}^{A}(x_i,k_i))_{0\leq i<l_{R+1}}}
\end{equation}
to denote the sum over all possible sequences of at most $l_{R+1}$ blocks $\tilde{B}_{R+1}^{A}(x_i,k_i)$ 
that can be crossed by a path $\Phi$. Since each block $B_{R+1}^{A}(x,k)$ that may be crossed by $\Phi$
lies in the union of (\ref{eq:allseq}), we may estimate the probability under the sum in (\ref{eq:sumbound}) 
from above by
\begin{equation}
\label{eq:upbound}
\sum_{(\tilde{B}_{R+1}^{A}(x_i,k_i))_{0\leq i<l_{R+1}}}
\P\Big(\begin{array}{ll}
&\mbox{the union in (\ref{eq:allseq}) contains at least $3A^{-R}\widetilde{\delta}_Rk_*/a_1^{d}a_2$ blocks}\\ 
&\mbox{$B_{R+1}^{A}(x,k)$ with $\chi^{A}(x,k)=1$ and $(x,k)\equiv ([x],[k])$}
\end{array}
\Big).
\end{equation}

Next, note that the union in (\ref{eq:allseq}) contains at most $l_{R+1}L^{d+1}$ $(R+1)$-blocks and that 
there are ${l_{R+1}L^{d+1}\choose n}$ ways of choosing $n$ blocks $B_{R+1}^{A}(x,k)$ with $\chi^{A}(x,k)=1$ 
from $l_{R+1}L^{d+1}$ $(R+1)$-blocks.  Hence, by the mixing condition in (\ref{wGartner})
for $A$ and $m$ large enough, each summand 
in \eqref{eq:upbound} is bounded from above by
\begin{equation}
\label{eq:binom}
\sum_{n=\delta_R k_{R+1}/a_1^{d}a_2}^{l_{R+1}L^{d+1}} \binom{l_{R+1}L^{d+1}}{n}(\rho_R)^n
\leq(1-\rho_R)^{-l_{R+1}L^{d+1}}
\P\big(T\geq 3A^{-R}\widetilde{\delta}_R k_*/a_1^{d}a_2\big),
\end{equation}
where $T=\mathrm{BINOMIAL}(l_{R+1}L^{d+1},\rho_R)$. Since
\begin{equation}
\E(T) = \rho_R l_{R+1}L^{d+1} = l_{R+1} = 3k_*/A^{R}L  
= 3A^{-R}A^{-4d(2d+1)R}k_*  = 3A^{-R}\widetilde{\delta}_R^{2}k_*
\ll 3A^{-R}\widetilde{\delta}_Rk_*,
\end{equation} 
we can apply standard large deviation estimates to bound the right-hand side of (\ref{eq:binom}). Indeed, 
by Bernstein's inequality, there is a $C'>0$ (depending on $a_1$ and $a_2$ only) such that for all $A$ 
and $m$ large enough,
\begin{equation}
\label{eq:bernstein}
\P\big(T\geq 3A^{-R}\widetilde{\delta}_R k_*/a_1^{d}a_2\big) 
\leq e^{-C'3A^{-R}\widetilde{\delta}_R k_*/a_1^{d}a_2}.
\end{equation}
Moreover, there is a $C''>0$ (not depending on $A$, provided $A$ is large enough) such that
\begin{equation}
\label{eq:rest}
(1-\rho_R)^{-l_{R+1}L^{d+1}} \leq e^{\rho_R l_{R+1}L^{d+1}/(1-\rho_R)}
\leq e^{C'' 3A^{-R}\widetilde{\delta}_R^2k_* }.
\end{equation}
Furthermore, by Lemma~\ref{lem:cardinality}, and after a possible increase of $C''$, the sum in (\ref{eq:allsum}) 
contains at most $e^{C''l_{R+1}}= e^{C''3A^{-R}\widetilde{\delta}_R^2k_*}$ elements. Hence, combining 
\eqref{eq:sumbound}, (\ref{eq:upbound}--\ref{eq:binom}) and (\ref{eq:bernstein}--\ref{eq:rest}), we see that 
the left-hand side of \eqref{eq:insblocks} is bounded from above by $e^{-CA^{-R}\widetilde{\delta}_R k_*}$, 
with $C$ such that $CA^{-R}\widetilde{\delta}_R k_* \geq (C'/a_1^{d}a_2-\widetilde{\delta}_R2C'')3A^{-R}
\widetilde{\delta}_Rk_*$, which yields the claim in \eqref{eq:insblocks}.
\epr


\subsection{Proof of Lemmas~\ref{lem:cardinality}--\ref{lem:crossedblocks}}
\label{S6.5}

\bpr
For the proof of Lemma~\ref{lem:cardinality}, see the proof of \cite[Claim 3.8]{EdHM12}.
The proof of Lemma~\ref{lem:crossedblocks} goes as follows. Let $R\in\N$. Divide time 
into intervals of length $LA^{R}$. Let $l_i^L$ and $l_i$ be the number of blocks $\widetilde{B}_R^{A}(x,k)$, 
respectively, $1$-blocks, crossed by $X^{\kappa}$ in the $i$-th time interval $[(i-1)LA^{R},iLA^{R})$, 
$1\leq i\leq t/LA^{R}$. Note that $l_i \geq LA^{R-1}$ because the length of the time-interval of 
each block $\widetilde{B}_R^{A}(x,k)$ is $LA^{R}$, which may be divided into $LA^{R-1}$ 
time-intervals of length $A$. Moreover $X^{\kappa}$ has to cross at least one $1$-block in 
each such interval of length $A$. Also note that if $l_i=LA^{R-1}$, then $l_i^L \leq 2 =2l_i/l_i
\leq 2l_i^{1}/LA^{R-1}$.  If $LA^{R-1}+1 \leq l_i \leq 2LA^{R-1}$, then $l_i^L\leq 3$, because 
$X^{\kappa}$ may start at an interface between two blocks $\widetilde{B}_R^{A}(x,k)$ and 
immediately jump from one such block to another. However, to afterwards reach the next block 
$\widetilde{B}_R^{A}(x,k)$ it has to cross at least $LA^{R-1}$ $1$-blocks, and so $l_i^{L}
\leq 3l_i/l_i\leq 3l_i^{1}/LA^{R-1}$. Furthermore, for $j \in \N$, if $jLA^{R-1}+1 \leq l_i \leq (j+1)
LA^{R-1}$, then 
\begin{equation}
\label{liest}
l_i^L \leq (j+2)l_i/l_i \leq (j+2)l_i/jLA^{R-1}.
\end{equation}
Therefore we have 
\begin{equation}
\label{eq:kR1}
k_* = \sum_{i=1}^{t/LA^{R-1}} l_{i} 
\geq \frac{LA^{R-1}}{3} \sum_{i=1}^{t/LA^{R-1}} l_{i}^L,
\end{equation}
or $\sum_{i=1}^{t/LA^{R-1}} l_{i}^L \leq (3/LA^{R-1}) k_* = l_R$, which completes the proof.
\epr


\section{Proof of Proposition \ref{prop:localTime}}
\label{S7}

In Section~\ref{S7.1} we reduce the problem to one dimension and recall 
two discrete rearrangement inequalities from the literature 
(Propositions~\ref{prop:rearrangeineq1}--\ref{prop:rearrangeineq2} below). 
In Section~\ref{S7.2} we use the latter to give the proof of Proposition~\ref{prop:localTime}.


\subsection{Reduction to one dimension and discrete rearrangement inequalities}
\label{S7.1}

\begin{lemma}
\label{lem:oned}
Let $B\subseteq \Z^d\times [0,t]$. Then, for all $C \geq 0$,
\begin{equation}
\label{eq:claimest}
E_0\big(e^{Cl_t(B)}\big) \leq E_0\big(e^{Cl_t(\pi_1(B))}\big).
\end{equation}
\end{lemma}
\begin{proof}
A $d$-dimensional simple random walk with jump rate $2d\kappa$ is a vector of $d$ independent
one-dimensional simple random walks, each having jump rate $2\kappa$. Hence, given any set 
$B\subseteq \Z^d \times[0,t]$,
\begin{equation}
\label{eq:reduction}
\forall\,s\geq 0\colon\qquad X^{\kappa}(s) \in B
\quad \Longrightarrow \quad \pi_1(X^{\kappa})(s) \in \pi_1(B).
\end{equation}
This in turn implies that $l_t(B) \leq l_t(\pi_1(B))$, which proves the claim.
\end{proof}

To prove Proposition~\ref{prop:localTime} we need two discrete rearrangement 
inequalities \cite{Pr96}, \cite{Pr98}. For an overview on continuous rearrangement 
inequalities we refer the reader to \cite[Chapter 3]{LL01}.

\begin{definition}
\label{ass:assumption}
A function $L\colon\,\Z\times\Z\to [0,\infty)$ is called of Riesz-type when, for all 
pairs of functions $f,g\colon\,\Z\rightarrow [0,\infty)$,
\begin{equation}
\label{eq:assumption}
\sum_{x,y\in\Z} f(x)L(x,y)g(y) 
\leq \sum_{x,y\in\Z} f^{\sharp}(x)L(x,y)g^{\sharp}(y).
\end{equation}
\end{definition}

\begin{proposition}{\rm \cite[Theorem 2.2]{Pr96}, \cite{Pr98})}
\label{prop:rearrangeineq1}
Let $K\colon\,[0,\infty) \to [0,\infty)$ be non-increasing. Then $L\colon\,\Z\times\Z \to [0,\infty)$ given by
$L(x,y)=K(|x-y|)$ is of Riesz-type.
\end{proposition}

Note that $(x,y) \mapsto p_s^{\kappa}(x,y)$ with $p_s^{\kappa}(x,y)$ the transition kernel 
of one-dimensional simple random walk with jump rate $2\kappa$ is of Riesz-type. Indeed, 
$p_s^{\kappa}(x,y) = p_s^{\kappa}(x-y,0) = p_s^{\kappa}(|x-y|,0)$ is a non-increasing function 
of $|x-y|$. 

The following multiple-sum version of Proposition~\ref{prop:rearrangeineq1} will be
needed also.

\begin{proposition}{\rm (\cite[Lemma 9.1 in Chapter 2]{Pr96}, \cite{Pr98})}
\label{prop:rearrangeineq2}
Fix $n\in\N$. Let $L_0, L_1, \ldots, L_{n-1}$ be a collection of Riesz-type functions 
on $\Z\times\Z$, and let $S_0, S_1, \ldots, S_n$ be a collection of non-negative functions 
on $\Z$. Then
\begin{equation}
\label{eq:rearrangeineq2}
 \sum_{x_0, x_1, \ldots, x_{n} \in \Z}
\left(\prod_{i=0}^{n-1} S_i(x_i) L_i(x_i,x_{i+1}) \right)S_n(x_n)
\leq \sum_{x_0, x_1, \ldots, x_{n} \in \Z}
\left(\prod_{i=0}^{n-1} S_i^\sharp(x_i) L_i(x_i,x_{i+1}) \right)S_n^\sharp(x_n).
\end{equation}
\end{proposition}


\subsection{Proof of Proposition \ref{prop:localTime}}
\label{S7.2}

\begin{proof}
Let $(B_R)_{R\in\N}$ be a sequence in $\Z \times [0,t]$ (recall Lemma~\ref{lem:oned}) and 
$(C_R)_{R\in\N}$ a sequence of nonnegative numbers.  Write
\begin{equation}
\label{eq:expansion}
E_0\Bigg(\exp\Bigg\{\sum_{R\in\N} C_R l_t(B_{R})\Bigg\}\Bigg)
=\sum_{n\in\N_0}\frac{1}{n!}
E_0\Bigg(\Bigg\{\sum_{R\in\N} C_R l_t(B_R)\Bigg\}^n\Bigg).
\end{equation}
The $n$-th moments in (\ref{eq:expansion}) may be rewritten as
\begin{equation}
\label{eq:expectation}
\sum_{R_1, \ldots, R_n \in\N}
\Bigg(\prod_{i=1}^{n}C_{R_i}\Bigg) E_0\Bigg(\prod_{i=1}^{n} l_t(B_{R_i})\Bigg).
\end{equation}
Write out
\begin{equation}
\label{eq:local}
\prod_{i=1}^{n} l_t(B_{R_i}) 
= \int_{0}^t ds_1 \ldots \int_0^t ds_n\,\,
\one\big\{X^{\kappa}(s_1)\in B_{R_1},\ldots X^{\kappa}(s_n) \in B_{R_n}\big\},
\end{equation}
so that the second factor under the sum in (\ref{eq:expectation}) equals
\begin{equation}
\label{eq:localexp}
\int_0^t ds_1 \ldots \int_0^t ds_n\,\, 
P_0\Big( X^{\kappa}(s_1) \in B_{R_1},\ldots X^{\kappa}(s_n) \in B_{R_n}\Big).
\end{equation}
Fix a choice of $(s_1, \ldots, s_n)\in [0,t]^n$, and let $B_{R_{s_i}}$ be the spatial part 
of $B_{R_i} \cap (\Z\times\{s_i\})$. Without loss of generality we may assume that 
$s_1< s_2 <\ldots < s_n$, so that the probability in (\ref{eq:localexp}) becomes
($x_0=0$, $s_0=0$)
\begin{equation}
\label{eq:proba}
\sum_{x_1, \ldots, x_n \in \Z}
\Bigg(\prod_{i=0}^{n} \one\Big\{x_i \in B_{R_{s_i}}\Big\}\Bigg)
\Bigg(\prod_{i=0}^{n-1} p_{s_{i+1}-s_i}^{\kappa}(x_i,x_{i+1})\Bigg).
\end{equation}
An application of Proposition \ref{prop:rearrangeineq2} gives that (\ref{eq:proba})
is bounded from above by
\begin{equation}
\label{eq:rearrange}
\sum_{x_1, \ldots x_n \in \Z}
\Bigg(\prod_{i=0}^{n} \one\Big\{x_i \in B^{\sharp}_{R_{s_i}}\Big\}\Bigg)
\Bigg(\prod_{i=0}^{n-1} p_{s_{i+1}-s_i}^{\kappa}(x_i,x_{i+1})\Bigg),
\end{equation}
so that, by (\ref{eq:localexp}),
\begin{equation}
\label{eq:localcomp}
E_0\Bigg(\prod_{i=1}^{n} l_t(B_{R_i})\Bigg)
\leq E_0\Bigg(\prod_{i=1}^{n} l_t(B^{\sharp}_{R_i})\Bigg).
\end{equation}
Inserting this back into (\ref{eq:expansion}) and (\ref{eq:expectation}), we get the claim.
\end{proof}


\section{Proof of Proposition \ref{prop:goodblocks}}
\label{S8}

In Section~\ref{S8.1} we introduce some notation and state two more propositions, 
Propositions~\ref{prop:counting}--\ref{prop:localcontrol} below, whose proof is given 
in Sections \ref{S8.3}--\ref{S8.4}. In Section~\ref{S8.2} we give the proof of 
Proposition~\ref{prop:goodblocks} subject to these propositions.


\subsection{Two more propositions}
\label{S8.1}

Henceforth we assume that $\alpha$ in \eqref{Bblocks} takes the form $\alpha =4M\kappa$ 
with $M$ a constant that will be determined later on. Recall the definition of $\pi_1$ below 
\eqref{PiL} and of $\bar\xi$ in \eqref{xibardef}.

\bd{pedestaldef}
The subpedestal of $B_1^{A,4M\kappa}(x,k)$ is (see Fig.~{\rm \ref{Defsufficient}})
\begin{equation}
\label{subpedestal}
\begin{aligned}
B_{1,\mathrm{sub}}^{A,4M\kappa}(x,k)=\Big\{&y \in \pi_1\big(B_1^{A,4M\kappa}(x,k)\big)\colon\,\\
& |y(j)-z(j)|\geq 2M\kappa A,\,j\in\{1,2,\dots,d\}\,\forall\,z\in
\partial\pi_1\big(B_1^{A,4M\kappa}(x,k)\big)\Big\}\times \{kA\}.
\end{aligned}
\end{equation}
\ed

\bd{sufficientblock}
Let $\varepsilon >0$, and $k,n\in\N_0$ such that $n\geq k$. A block $B_1^{A,4M\kappa}(x,k)$ 
is called $\varepsilon$-sufficient at level $n$ when, for every $y \in \pi_1(B_{1,\mathrm{sub}}^{A,
4M\kappa}(x,k))$,
\begin{equation}
\label{good}
E_y\Bigg(\exp\bigg\{\int_{0}^{A}\overline{\xi}(X^{\kappa}(s),A(n-k)-s)\, ds\bigg\}
\one{\big\{N(X^{\kappa},A)\leq M\kappa A\big\}}\Bigg) 
\leq e^{\varepsilon A}.
\end{equation}
Otherwise it is called $\varepsilon$-insufficient at level $n$. A subpedestal is called 
$\varepsilon$-(in)sufficient at level $n$ when its corresponding block is 
$\varepsilon$-(in)sufficient at level $n$.
\ed

\begin{figure}[htbp]
\begin{center}
\setlength{\unitlength}{0.35cm}
\begin{picture}(22,15)(0,2)
\put(-1,5){\vector(1,0){22}}
\put(-1,5){\vector(0,1){11}}
\put(7,15){\line(1,0){7}}
\put(14,15){\line(0,-1){9}}
\put(14,6){\line(-1,0){7}}
\put(7,6){\line(0,1){2.25}}
\put(7,12.75){\line(0,1){2.25}}
\linethickness{1mm}
\put(7,8.25){\line(0,1){4.5}}
\put(22,4.8){time}
\put(-2,17){space}
\put(7.5,9){$B_1^{A,4M\kappa}(x,k)$}
\put(6.3,3.5){$kA$}
\put(12.5,3.5){$(k+1)A$}
\put(-10,5.7){$(x(j)-1)4M\kappa A$}
\put(-10,14.5){$(x(j)+1)4M\kappa A$}
\end{picture}
\caption{\small The thick line is the subpedestal.}
\label{Defsufficient}
\end{center}
\end{figure}
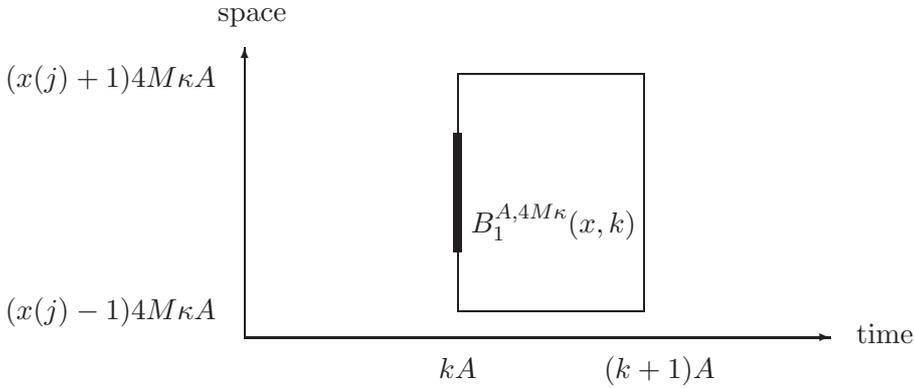

\begin{proposition}
\label{prop:counting}
Let $A>1$. There is a constant $C_3>0$ such that for all $n\in\N$ the number of different 
sequences of subpedestals $B_{1,\mathrm{sub}}^{A,4M\kappa}(0,0), 
B_{1,\mathrm{sub}}^{A,4M\kappa}(x_1,1),\ldots, B_{1,\mathrm{sub}}^{A,4M\kappa}(x_{n-1},n-1)$ 
with the property that there is a path $\Phi\colon\,[0,An]\to\Z^d$ with at most $M\kappa An$ 
jumps satisfying $\Phi(kA)\in B_{1,\mathrm{sub}}^{A,4M\kappa}(x_k,k)$, $k\in\{0,1,\ldots,n-1\}$,
is bounded from above by $e^{C_3n}$.
\end{proposition}

\begin{proposition}
\label{prop:localcontrol}
Fix $\varepsilon>0$. Let $\delta=\tfrac14\varepsilon$ in the definition of $\overline{\xi}$ and $A>1$. Then there is a $\kappa_0>0$ 
such that, for all $\kappa\geq\kappa_0$ and $\xi$-a.s.\ for all $n\in\N$, all blocks $B_{1}^{A,4M\kappa}
(x,k)$, $x\in\Z^d$, $k\in\N$, $k\leq n$, are $\varepsilon$-sufficient at level $n$.
\end{proposition}


\subsection{Proof of Proposition \ref{prop:goodblocks} subject to two propositions}
\label{S8.2}

\begin{proof}
The proof comes in 2 Steps.

\medskip\noindent
{\bf 1.}
Fix $\varepsilon>0$ and put $\delta=\tfrac14\varepsilon$. Choose $\kappa\geq\kappa_0$ according 
to Proposition~\ref{prop:localcontrol}. Then the tail estimate $P\big(\mathrm{POISSON}(\lambda)
\geq k\big) \leq e^{-\lambda}(\lambda e)^k/k^k$, $k\geq 2\lambda +1$, for Poisson-distributed random 
variables with mean $\lambda$ shows that, for $M>0$ large enough,
\begin{equation}
\label{eq:outerpart}
\begin{aligned}
E_0&\Bigg(
\exp\bigg\{\int_0^{An}\overline{\xi}(X^{\kappa}(s),An-s)\, ds\bigg\}
\one\big\{N(X^{\kappa},An)>M\kappa An\big\}\Bigg)\\
&\leq e^{2\delta A^{d+1}n} e^{-2d\kappa An} \exp\big\{-M\kappa An[\log(M/2d)-1]\big\},
\end{aligned}
\end{equation}
where we use \eqref{xibardef}. Since we later let $\kappa\to\infty$, (\ref{eq:outerpart}) shows that it 
is enough to concentrate on contributions coming from paths with at most $M\kappa An$ jumps. To 
that end, fix a $\Z^d$-valued sequence of vertices $x_0,x_1,\ldots, x_{n-1}$ such that $x_0=0$ and 
such that there is a path that starts in $0$, makes $0\leq j\leq M\kappa An$ jumps, and is 
in the subpedestal $B_{1,\mathrm{sub}}^{A,4M\kappa}(x_k,k)$ at time $kA$ for $k\in\{0,1,\ldots, n-1\}$. 
By the Markov property of $X^{\kappa}$ applied at times $kA$, $k\in\{1,2,\ldots, n-1\}$,
\begin{equation}
\label{eq:timepart}
\begin{aligned}
&E_0\Bigg(
\exp\bigg\{\int_0^{An}\overline{\xi}(X^{\kappa}(s),An-s)\, ds\bigg\}
\one\big\{N(X^{\kappa},An)\leq M\kappa An\big\}\\
&\qquad\qquad\times \prod_{k=1}^{n-1}
\one\Big\{X^{\kappa}(kA)\in B_{1,\mathrm{sub}}^{A,4M\kappa}(x_k,k) \Big\}\Bigg)\\
&\quad \leq \prod_{k=0}^{n-1} \sup_{y\in \pi_1\big(B_{1,\mathrm{sub}}^{A,4M\kappa}(x_k,k)\big)}
E_{y}\Bigg(\exp\bigg\{\int_{0}^{A}\overline{\xi}(X^{\kappa}(s),A(n-k)-s)\, ds\bigg\}\Bigg).
\end{aligned}
\end{equation}
This is at most
\begin{equation}
\label{eq:Srewrite}
\begin{aligned}
&\prod_{k=0}^{n-1}\Bigg[\sup_{y\in \pi_1\big(B_{1,\mathrm{sub}}^{A,4M\kappa}(x_k,k)\big)} 
E_{y}\bigg(\exp\Big\{\int_{0}^{A}\overline{\xi}(X^{\kappa}(s),A(n-k)-s)\, ds\Big\}
\one\big\{N(X^{\kappa}, A)\leq M\kappa A\big\}\bigg)\\
&\quad +\sup_{y\in \pi_1\big(B_{1,\mathrm{sub}}^{A,4M\kappa}(x_k,k)\big)} 
E_{y}\bigg(\exp\Big\{\int_{0}^{A}\overline{\xi}(X^{\kappa}(s),A(n-k)-s)\, ds\Big\}
\one\big\{N(X^{\kappa},A) >M\kappa A\big\}\bigg)\Bigg]\\
&= \sum_{J\subset \{0,1,\ldots,n-1\}}
\Bigg[\\
&\quad\prod_{k\in J}
\sup_{y\in \pi_1(B_{1,\mathrm{sub}}^{A,4M\kappa}(x_k,k))}E_{y}\Bigg(
\exp\bigg\{\int_{0}^{A}\overline{\xi}(X^{\kappa}(s),A(n-k)-s)\, ds\bigg\}
\one\big\{N(X^{\kappa}, A)\leq M\kappa A\big\}\Bigg)\\
&\times\prod_{k\notin J}\sup_{y\in \pi_1(B_{1,\mathrm{sub}}^{A,4M\kappa}(x_k,k))} E_{y}\Bigg(
\exp\bigg\{\int_{0}^{A}\overline{\xi}(X^{\kappa}(s),A(n-k)-s)\, ds\bigg\}
\one\big\{N(X^{\kappa},A) >M\kappa A\big\}\Bigg)\Bigg].
\end{aligned}
\end{equation}
Now, by the Poisson tail estimate mentioned above and the fact that $\overline{\xi}< 2\delta A^d$, 
the second factor under the sum in (\ref{eq:Srewrite}) may be bounded from above by
\begin{equation}
\label{eq:toomanyjumps}
\Big(e^{2\delta A^{d+1}}e^{-2d\kappa A}\exp\big\{-M\kappa A[\log(M/2d)-1]\big\}\Big)^{n-|J|}.
\end{equation}
Since, by Proposition \ref{prop:localcontrol} and our choice of $\kappa$ (see the observation 
made prior to (\ref{eq:outerpart})), all blocks $B_{1}^{A,4M\kappa}(x,k)$, $x\in\Z^d$, $k\in\N_0$, 
$k\leq n$, are $\varepsilon$-sufficient at level $n$, we may conclude that all $y\in \pi_1
(B_{1,\mathrm{sub}}^{A,4M\kappa}(x_k,k))$ with $k\in J$ are in an $\varepsilon$-sufficient 
subpedestal at level $n$. Hence, using the binomial formula, we may estimate (\ref{eq:Srewrite}) 
from above by
\begin{equation}
\label{eq:Sbound}
\begin{aligned}
\sum_{J\subset \{0,1,\ldots,n-1\}} 
&e^{A\varepsilon |J|} \Big(e^{2\delta A^{d+1}}e^{-2d\kappa A}
\exp\big\{-M\kappa A[\log(M/2d)-1]\big\}\Big)^{n-|J|}\\
&= \Big(e^{A\varepsilon}+ e^{2\delta A^{d+1}}e^{-2d\kappa A}
\exp\big\{-M\kappa A[\log(M/2d)-1]\big\}\Big)^n.
\end{aligned}
\end{equation}

\medskip\noindent
{\bf 2.}
Summing over all possible sequences $(x_i)_{i\in\{1,2,\ldots,n-1\}}$ compatible with a path $\Phi$ 
such that $\Phi(0) = 0$ and $N(\Phi,An)\leq M\kappa An$, and using (\ref{eq:timepart}--\ref{eq:Sbound}) 
and Proposition~\ref{prop:counting}, we obtain
\begin{equation}
\label{eq:lastsums}
\begin{aligned}
&E_0\Bigg(\exp\bigg\{\int_0^{An}\overline{\xi}(X^{\kappa}(s),An-s)\, ds\bigg\}
\one\big\{N(X^{\kappa},An)\leq M\kappa An\big\}\Bigg)\\
&\leq \sum_{x_1,x_2,\ldots, x_{n-1}}
E_0\Bigg(\exp\bigg\{\int_0^{An}\overline{\xi}(X^{\kappa}(s),An-s)\, ds\bigg\}
\one\big\{N(X^{\kappa},An) \leq M\kappa An\big\}\\
&\qquad\qquad\qquad\qquad\quad\times\prod_{k=1}^{n-1}
\one\Big\{X^{\kappa}(kA)\in B_{1,\mathrm{sub}}^{A,4M\kappa}(x_k,k)\Big\}\Bigg)\\
&\leq \sum_{x_1,x_2,\ldots, x_{n-1}}
\prod_{k=0}^{n-1} \sup_{y\in  B_{1,\mathrm{sub}}^{A,4M\kappa}(x_k,k)}
E_y\Bigg(\exp\bigg\{\int_0^{An}\overline{\xi}(X^{\kappa}(s),An-s)\, ds\bigg\}
\Bigg)\\
&\leq e^{C_3n}\Big(e^{A\varepsilon}+ e^{2\delta A^{d+1}}e^{-2d\kappa A}
\exp\big\{-M\kappa A[\log(M/2d)-1]\big\}\Big)^n.
\end{aligned}
\end{equation}
Combining (\ref{eq:outerpart}--\ref{eq:lastsums}), we get
\begin{equation}
\label{eq:takinglimit1}
\begin{aligned}
&\limsup_{\kappa\to\infty} \limsup_{n\to\infty}
\frac{1}{An} \log E_0\Bigg(\exp\bigg\{\int_0^t\overline{\xi}(X^{\kappa}(s),t-s)\, ds\bigg\}\Bigg)\\
&\leq \frac{C_3}{A} +\frac{1}{A}\limsup_{\kappa\to\infty}
\log\Big(e^{A\varepsilon}+ e^{2\delta A^{d+1}}e^{-2d\kappa A}
\exp\big\{-M\kappa A[\log(M/2d)-1]\big\}\Big)
= \frac{C_3}{A} + \varepsilon.
\end{aligned}
\end{equation}
Since $\varepsilon=4\delta$, this yields the claim.
\end{proof}


\subsection{Proof of Proposition \ref{prop:counting}}
\label{S8.3}

\begin{proof}
Write $\|\cdot\|$ for the $\ell^1$-norm on $\Z^d$. Let $B_{1,\mathrm{sub}}^{A,4M\kappa}(0,0), 
B_{1,\mathrm{sub}}^{A,4M\kappa}(x_1,1),\ldots,B_{1,\mathrm{sub}}^{A,4M\kappa}(x_{n-1},$ $n-1)$
be a sequence of subpedestals that may be crossed by a path $\Phi$ with at most $M\kappa An$ 
jumps. Since $\Phi$ needs at least $(\|x_k-x_{k-1}\|-d)_{+}4M\kappa A$ jumps to go from 
$B_{1,\mathrm{sub}}^{A,4M\kappa}(x_{k-1},k-1)$ to $B_{1,\mathrm{sub}}^{A,4M\kappa}(x_{k},k)$, 
$k\in\{1,2,\ldots,n-1\}$, we obtain the bound
\begin{equation}
\label{eq:differencesum1}
\sum_{k=1}^{n-1}(\|x_k-x_{k-1}\|-d)_+ \leq \frac{M\kappa An}{4M\kappa A} = \frac{n}{4},
\end{equation}
which implies that
\begin{equation}
\label{eq:differencesum2}
\sum_{k=1}^{n-1} \|x_k-x_{k-1}\| \leq \frac{(1+4d)n}{4}.
\end{equation}
As shown in Hardy and Ramanujan~\cite{HR18} and Erd\"os~\cite{E42}, there are $a,b>0$ such 
that the number of integer-valued sequences $(a_k)_{k\in\N}$ such that $\sum_{k\in\N} a_k \leq 
(1+4d)n/4$ is bounded from above by $ane^{b\sqrt{n}}$. To conclude, define $a_k = \|x_k-x_{k-1}\|$, 
$k\in\{1,2,\ldots, n-1\}$, and note that the sequence $(a_k)_{k\in\{1,2,\ldots,n-1\}}$ determines the 
sequence $(x_k)_{k\in\{0,1,\ldots,n-1\}}$ uniquely when it is known for all $k\in\{1,2,\ldots,n-1\}$ 
and all $j\in\{1,2,\ldots,d\}$ whether $x_k(j)-x_{k-1}(j)$ is positive, zero or negative. Consequently, 
the number of different subpedestals $B_{1,\mathrm{sub}}^{A,4M\kappa}(0,0), 
B_{1,\mathrm{sub}}^{A,4M\kappa}(x_1,1),\ldots,B_{1,\mathrm{sub}}^{A,4M\kappa}(x_{n-1},n-1)$ 
that may be crossed by a path $\Phi$ with at most $M\kappa An$ jumps is bounded from above by 
$3^{dn}ane^{b\sqrt{n}} \leq e^{C_3n}$ for some $C_3>0$.
\end{proof}


\subsection{Proof of Proposition \ref{prop:localcontrol}}
\label{S8.4}

The proof of Proposition~\ref{prop:localcontrol} is given in Section \ref{S8.4.3} subject to
Lemmas~\ref{lem:FKestimate}--\ref{lem:eigenvalue} below, which are stated in 
Sections~\ref{S8.4.1}--\ref{S8.4.2}. The proof of the first lemma is given in Section~\ref{S8.4.1}, 
the proof of the second lemma is deferred to Appendix~\ref{Appendix*}. 


\subsubsection{A time-dependent Feynman-Kac estimate}
\label{S8.4.1}

Recall \eqref{xibardef}. Abbreviate
\begin{equation}
\label{eq:int}
Q^{\kappa\log\kappa} = (-\kappa\log\kappa,\kappa\log\kappa)^d\cap\Z^d.
\end{equation} 

\begin{lemma}
\label{lem:FKestimate}
Fix $A>1$ and $m>0$ such that $Am\in\N$. There is a $\kappa_0=\kappa_0(M,A)$ 
such that, $\xi$-a.s.\ for all $x\in\Z^d$,
\begin{equation}
\label{eq:timedepFK}
\log E_x\Bigg(\exp\bigg\{
\int_0^{A}\overline{\xi}(X^{\kappa}(s),A-s)\, ds\bigg\}
\one\big\{N(X^{\kappa},A)\leq M\kappa A\big\}\Bigg)
\leq \frac{\kappa}{m}\sum_{k=1}^{Am}\lambda_{1}(\overline{\xi}_k/\kappa),
\quad \kappa\geq\kappa_0,
\end{equation}
where $\lambda_1(\overline{\xi}_k/\kappa)$ is the top of the spectrum of the operator
$\Delta + \frac{1}{\kappa} \sup_{r\in [(k-1)/m,k/m)}\overline{\xi}(\cdot,r)$, $k\in\{1,2,\ldots,Am\}$.
\end{lemma}

\begin{proof}
We give the proof for $x=0$. The proof for $x\in\Z^d\backslash\{0\}$ goes along the same lines.
First note that we may rewrite the expectation in the left-hand side of (\ref{eq:timedepFK}) as
\begin{equation}
\label{eq:rewrite}
 E_0\Bigg(\exp\bigg\{
\frac{1}{\kappa}\int_0^{\kappa A}\overline{\xi}(X(s),A-s/\kappa)\, ds\bigg\}
\one\big\{N(X,\kappa A)\leq M\kappa A\big\}\Bigg),
\end{equation}
where $X$ is simple random walk with step rate $2d$. Furthermore, there is a $\kappa_0 =\kappa_0(M,A)$ 
such that $M\kappa A\leq \kappa\log\kappa$ for all $\kappa\geq \kappa_0$. Hence, by the Markov property 
of $X$ applied at times $k\kappa/m$, $k\in\{1,2,\ldots, Am\}$, we may estimate (\ref{eq:rewrite}) from above 
by
\begin{equation}
\label{eq:Markovbound}
\begin{aligned}
E_0&\Bigg(\exp\bigg\{\frac{1}{\kappa}\int_0^{\kappa A}
\overline{\xi}(X(s),A-s/\kappa)\, ds\bigg\}
\one\big\{X([0,\kappa A])\subseteq Q^{\kappa\log\kappa}\big\}\Bigg)\\
&\leq \prod_{k=1}^{Am} \sup_{ \substack{x\in\Z^d \\ \|x\|_\infty< \kappa\log\kappa} }
E_x\Bigg(\exp\bigg\{\frac{1}{\kappa}\int_0^{\kappa /m}
\overline{\xi}(X(s),k/m- s/\kappa)\, ds\bigg\}
\one\big\{X([0,\kappa/m])\subseteq Q^{\kappa\log\kappa}\big\}\Bigg).
\end{aligned}
\end{equation}
Next, for $k\in\{1,2,\ldots, Am\}$ define
\begin{equation}
\label{eq:xin}
\overline{\xi}_k(x) = \sup_{r\in [(k-1)/m,k/m)} \overline{\xi}(x,r), \qquad x\in\Z^d.
\end{equation}
Then (\ref{eq:Markovbound}) is at most
\begin{equation}
\label{eq:upperxi}
\prod_{k=1}^{Am} \sup_{ \substack{x\in\Z^d // \|x\|_\infty< \kappa\log\kappa} }
E_x\Bigg(\exp\bigg\{\frac{1}{\kappa}\int_0^{\kappa/m}
\overline{\xi}_k(X(s))\, ds\bigg\}
\one\{X([0,\kappa /m])\subseteq Q^{\kappa\log\kappa}\}\Bigg).
\end{equation}
For each $k\in\{1,2,\ldots, Am\}$, each expectation under the product in (\ref{eq:upperxi})
is a solution of the equation
\begin{equation}
\label{eq:PAMDirichlet}
\begin{aligned}
\begin{cases}
\frac{\partial u_k}{\partial t}(x,t) = \left[(\Delta + \frac{1}{\kappa}
\overline{\xi}_k(x))u_k\right](x,t),\\
u_k(x,0) = 1,
\end{cases}
\qquad \|x\|_\infty<\kappa\log\kappa, t\geq 0,
\end{aligned}
\end{equation}
with Dirichlet boundary conditions evaluated at time $\kappa/m$. However, on any finite
subset of $\Z^d$ the operator $\Delta +\frac{1}{\kappa}\overline{\xi}_k$ is a self-adjoint 
matrix. Therefore, by the spectral representation theorem, we may rewrite each expectation 
under the product in \eqref{eq:upperxi} as
\begin{equation}
\label{eq:spectralrep}
\sum_{j=1}^{|Q^{\kappa\log\kappa}|}
e^{(\kappa/m)\lambda_{j}^{D}(\overline{\xi}_k/\kappa)}
\langle v_{j}^{k},\one_{Q^{\kappa\log\kappa}}\rangle \,v_{j}^{k}(x),
\end{equation}
where $\lambda_{j}^{D}(\overline{\xi}_k/\kappa)$ is the $j$-th largest eigenvalue of 
$\Delta+\overline{\xi}_k/\kappa$ with Dirichlet boundary conditions on $Q^{\kappa\log\kappa}$, 
$j\in\{1,2,\ldots,|Q^{\kappa\log\kappa}|\}$, and the $v_j^{k}$, $j\in\{1,2,\ldots,|Q^{\kappa\log\kappa}|\}$, 
form an orthonormal system of eigenvectors such that, for all $k\in\{1,2,\ldots, Am\}$,
\begin{equation}
\label{eq:spectralthm}
\R^{|Q^{\kappa\log\kappa}|} = \mathrm{ker}\big(e^{\Delta+\overline{\xi}_k/\kappa}\big)
\oplus \mathrm{span}\big\{v_j^{k},j\in\{1,2,\ldots,|Q^{\kappa\log\kappa}|\}\big\} .
\end{equation}
(Since $e^{\Delta+\overline{\xi}_k/\kappa}$ is a strictly positive operator, $\mathrm{ker}
(e^{\Delta+\overline{\xi}_k/\kappa})=\{0\}$.) In particular, for each $k\in\{1,2,
\ldots, Am\}$ there is a sequence of real-valued numbers $(\mu_{j}^{k})_{j\in\{1,2,\ldots,|Q^{\kappa
\log\kappa}|\}}$ such that $\sum_{j=1}^{k}\mu_{j}^{k}v_j^{k} = \one_{Q^{\kappa\log\kappa}}$.
Inserting the above representation of $\one_{Q^{\kappa\log\kappa}}$ into (\ref{eq:spectralrep}),
we see that (\ref{eq:spectralrep}) is bounded from above by
\begin{equation}
\label{eq:spectralest}
e^{(\kappa/m)\lambda_1^{D}(\overline{\xi}_k/\kappa)}.
\end{equation}
Combining (\ref{eq:rewrite})--(\ref{eq:spectralest}), we get
\begin{equation}
\label{eq:spectralbound}
\log E_x\Bigg(\exp\bigg\{
\int_0^{A}\overline{\xi}(X^{\kappa}(s),A-s)\, ds\bigg\}
\one\big\{N(X^{\kappa};A)\leq M\kappa A\big\}\Bigg)
\leq \frac{\kappa}{m}\sum_{k=1}^{Am} \lambda_1^{D}(\overline{\xi}_k/\kappa).
\end{equation}
Finally, by the Rayleigh-Ritz principle we have that $\lambda_1^{D}(\overline{\xi}_k/\kappa)\leq
\lambda_1(\overline{\xi}_k/\kappa)$, where $\lambda_1(\overline{\xi}_k/\kappa)$ is the top of 
the spectrum of $\Delta + \overline{\xi}_k/\kappa$.
\end{proof}


\subsubsection{A spectral estimate}
\label{S8.4.2}

Let $(B(x))_{x\in\Z^d}$ be an arbitrary partition of $\Z^d$ into finite boxes. Let $\langle\cdot,\cdot\rangle$ 
be the scalar product on $\R^{B}$ and on $\ell^2(\Z^d)$. Let $V\colon\,\Z^d\to\R$ be bounded such that 
there is a $\delta>0$ for which
\begin{equation}
\label{eq:V}
\frac{1}{|B(x)|} \sum_{y\in B(x)} V(y) \leq 2\delta \qquad x\in\Z^d.
\end{equation}

The proof of the following lemma is deferred to Appendix~\ref{Appendix*}.

\begin{lemma}
\label{lem:eigenvalue}
Subject to \eqref{eq:V}, there is a $\kappa_0>0$ such that, for all $\kappa\geq\kappa_0$,
\begin{equation}
\label{eq:eigenvalue}
\sup_{\substack{f\in l^2(\Z^d)\\\|f\|_2=1}} 
\Big\langle \Big(\Delta +\frac{1}{\kappa} V\Big)f,f\Big\rangle \leq 4\frac{1}{\kappa}\delta.
\end{equation}
\end{lemma}

\noindent
Lemma~\ref{lem:eigenvalue} and the Rayleigh-Ritz principle yield that the top of the spectrum 
of $\Delta +\frac{1}{\kappa} V$ is bounded from above by $4\frac{1}{\kappa}\delta$ for 
$\kappa\geq\kappa_0$.


\subsection{Completion of the proof of Proposition \ref{prop:localcontrol}}
\label{S8.4.3}

\begin{proof}
Fix $\delta >0$, $A>1$ and $m>1$. By Lemma~\ref{lem:FKestimate}, there is a $\kappa_0>0$ such 
that, for all $\kappa\geq\kappa_0$ and $\xi$-a.s.\ for all $x\in\Z^d$,
\begin{equation}
\label{eq:applicationFKest}
\begin{aligned}
\log &E_x\Bigg(\exp\bigg\{
\int_0^{A}\overline{\xi}(X^{\kappa}(s),A-s)\, ds\bigg\}
\one\big\{N(X^{\kappa};A)\leq M\kappa A\big\}\Bigg)\\
&\leq \frac{\kappa}{m} \sum_{k=1}^{Am} \lambda_{1}(\overline{\xi}_k/\kappa),
\qquad \kappa\geq\kappa_0=\kappa_0(M,A).
\end{aligned}
\end{equation}
Next, by Lemma~\ref{lem:eigenvalue} with $V=\overline{\xi}_k$, $k\in\{1,2,\ldots,Am\}$ (recall
(\ref{eq:xin})) and $B(x) = \pi_1(B_{1}^{A}(x,0))$ (recall (\ref{Bblocks}); $\pi_1$ denotes 
the projection onto the spatial coordinates), there is a $\kappa_1>0$ such that, for all 
$\kappa\geq\kappa_1$ and all $k\in\{1,2,\ldots,Am\}$,
\begin{equation}
\label{eq:applicationeigenvalueest}
\lambda_1(\overline{\xi}_k/\kappa)\leq 4\frac{1}{\kappa}\delta.
\end{equation}
This shows that, $\xi$-a.s.\ for $\kappa\geq\max\{\kappa_0,\kappa_1\}$, any block $B_{1}^{A,4M\kappa}
(x,0)$, $x\in\Z^d$, is $\varepsilon$-sufficient at level $1$. The stationarity of $\xi$ in time completes the 
proof.
\end{proof}


\appendix


\section{Proof of Lemmas \ref{lem:LDPest}--\ref{lem:crossing}}
\label{Appendix}

In this section we prove two lemmas that were used in Section~\ref{S5}. 


\subsection{Proof of Lemma \ref{lem:LDPest}}
\label{A1}

\begin{proof}
Our first observation is that
\begin{equation}
\label{eq:LDP}
\limsup_{t\to\infty} \frac{1}{t} \log 
E_0\left(\exp\left\{\sum_{i=1}^{n}\beta_i \sum_{x \in I_i}
l_t(X^\kappa,x)\right\}\right) \leq \mu,
\end{equation}
where
\begin{equation}
\label{eq:varrep}
\mu = \sup_{\substack{ f\in l^2(\Z)\colon\\ \|f\|_2=1,\,f\geq 0}} \mu(f)
\end{equation}
with
\begin{equation}
\mu(f) = \left(f,\left[\kappa\Delta+\sum_{i=1}^n \beta_i\one_{I_i}\right]f\right)
= \sum_{i=1}^{n}\beta_i \sum_{x\in I_i}f^2(x) -\kappa \sum_{x\in\Z} [f(x+1)-f(x)]^2. 
\end{equation}
Indeed, this follows from the large deviation principle for the occupation time measure
of one-dimensional simple random walk on $\Z$ with jump rate $2\kappa$ (which is 
the continuous-time Markov process with generator $\kappa\Delta$) in combination 
with Varadhan's lemma (see \cite[Chapters 3--4]{dH00}). A formal proof proceeds by 
truncating $\Z$ to a large finite torus, wrapping the random walk around the torus, 
deriving the claim for a fixed torus size, letting the torus  size tend to infinity, and 
showing that the variational formula on the finite torus converges to the variational 
formula on $\Z$. The details are standard and are left to the reader (see 
\cite[Chapter 8]{dH00}).

We claim that $\mu$ is the largest eigenvalue of $\kappa\Delta + \sum_{i=1}^{n}\beta_{i}\one_{I_i}$. 
Indeed, by \cite[Theorem 2.2]{HMO11} the operator $\kappa\Delta + \sum_{i=1}^{n}\beta_{i}\one_{I_i}$ 
has at least one eigenvalue. Consequently, the min-max principle \cite[Theorem XIII.1]{RS4} yields the 
claim. The inequality in \eqref{eq:localtimelem} now follows from \cite[Corollary 1.4]{S10}. 
\end{proof}


\subsection{Proof of Lemma \ref{lem:crossing}}
\label{A2}

\begin{proof}
Let $t,\kappa >0$ and let $X^{\kappa}$ be one-dimensional random walk with step rate $2\kappa >0$.
We first show that for all $C>0$,
\begin{equation}
\label{eq:RWsupremum}
P_0\bigg(\sup_{0\leq s\leq t} |X^{\kappa}(s)|\geq C\sqrt{\kappa t}\bigg)
\leq 2e^{-\frac{C^2\sqrt{\kappa t}}{2(C+3\sqrt{\kappa t})}}.
\end{equation}
The proof is based on a discretization argument in combination with Bernstein's inequality. Fix $n\in\N$ 
with $n\gg \kappa$, and define
\begin{equation}
\label{eq:q}
q^n(y) = 
\begin{cases}
1-\frac{2\kappa}{n}, &y=0 \\
\frac{\kappa}{n},  &y=\pm 1 \\
0, &\mbox{elsewhere}.
\end{cases}
\end{equation}
Let $X^{(n)}=(X^{(n)}(t))_{t\geq 0}$ be the discrete-time random walk with jump distribution $q^n$ 
and jump times $k/n$, $k\in\N$. Then, for each $t>0$, $(X^{(n)}(s))_{0\leq s\leq t}$ converges weakly 
as $n\to\infty$ to $(X^{\kappa}(s))_{0\leq s\leq t}$ in the Skorokhod space $D([0,t],\Z)$.  Since $X^{(n)}$ 
is unlikely to move in a short time interval, uniformly in $n$, it is enough to prove (\ref{eq:RWsupremum}) 
for $X^{(n)}$ with $n$ fixed. To that end, let $k\in\N$ be such that $k/n\leq t< (k+1)/n$. Note that, because 
$X^{(n)}$ is a martingale, Doob's maximal inequality and Bernstein's inequality yield
\begin{equation}
\label{eq:supest}
P_0\bigg(\sup_{0\leq s\leq t} X^{(n)}(s)\geq C\sqrt{\kappa t}\bigg)
\leq \exp\bigg\{-\frac{C^2\kappa t}{2(C\sqrt{\kappa t} + 3\kappa t)}\bigg\}
\qquad \forall\ C>0.
\end{equation}
The same inequality is valid for the probability of $\sup_{0\leq s\leq t}[ -X^{(n)}(s)]\geq C\sqrt{\kappa t}$, 
which yields the claim in \eqref{eq:RWsupremum}.

Next, note that if $X^{\kappa}$ leaves the spatial part of a space-time block $B_1(x,k;\kappa)$, then 
there is at least one coordinate $j\in\{1,\ldots,d\}$ such that $\pi_j(X^{\kappa}(s)) \notin [\sqrt{\kappa}x(j)A,
\sqrt{\kappa}(x(j)+1)A)$ for some $s \in [kA, (k+1)A)$, where $\pi_j(X^{\kappa})$ denotes the projection 
of $X^{\kappa}$ onto the $j$-th coordinate. In particular, if $X^{\kappa}$ visits $l_i^{\kappa}$ 
$(\kappa,1)$-blocks with $l_i^{\kappa}>d$ in the time interval $[(i-1)A,iA)$, then there is at least 
one coordinate that visits at least $l_i^{\kappa}/d$ one-dimensional $(\kappa,1)$-blocks, i.e., blocks 
of the form $[\sqrt{\kappa}xA,\sqrt{\kappa}(x+1)A)\times [kA,(k+1)A)$, $x\in\Z$. Consequently, without 
loss of generality we may assume that $X^{\kappa}$ is one-dimensional simple random walk with step 
rate $2\kappa$. Given $l_1^{\kappa},\ldots, l_{t/A}^{\kappa}\in\N$, we say that $X^{\kappa}$ has label 
$(l_1^{\kappa},\ldots, l_{t/A}^{\kappa})$ when $X^{\kappa}$ crosses $l_i^{\kappa}$ $(\kappa,1)$-blocks 
in the time interval $[(i-1)A,iA)$, $1\leq i\leq t/A$. 

Next, fix $C_7>0$ and let $k_*(\kappa)\geq C_7t$, and note that
\begin{equation}
\label{eq:labelest}
P_0\Big(X^{\kappa} \mbox{ crosses $k_*(\kappa)$ $(\kappa,1)$-blocks}\Big)
= \sum_{\substack{(l_i^{\kappa})_{i=1}^{t/A}\colon \\ \sum l_i^{\kappa} = k_*(\kappa)}}
P_0\Big(X^{\kappa} \mbox{ has label $\big(l_1^{\kappa},\ldots, l_{t/A}^{\kappa}\big)$}\Big).
\end{equation}
Using the Markov property and the fact that a path crossing $l_i^{\kappa}$ 
$(\kappa,1)$-blocks has to travel a distance at least $(l_i^{\kappa}-2)\sqrt{\kappa}A/2$,
we may further estimate each summand in the right-hand side of (\ref{eq:labelest}) by
\begin{equation}
\label{eq:Markovapp}
\begin{aligned}
&E_0\Big(\one\{X^{\kappa} \mbox{ has label $(l_1^{\kappa},\ldots, l_{t/A-1}^{\kappa})$}\}
P_{X^{\kappa}(t-A)}\big(X^{\kappa} \mbox{ has label $l_{t/A}^{\kappa}$}\big)\Big)\\
&\leq E_0\bigg(\one\{X^{\kappa} \mbox{ has label $(l_1^{\kappa},\ldots, l_{t/A-1}^{\kappa})$}\}
P_{X^{\kappa}(t-A)}\Big(\sup_{0\leq s\leq A}|X^{\kappa}(s)|\geq (l_i^{\kappa}-2)\sqrt{\kappa}A/2\Big)\bigg).
\end{aligned}
\end{equation}
To proceed, note that, by \eqref{eq:RWsupremum},
\begin{equation}
\label{eq:probaest}
P_{X^{\kappa}(t-A)} \Big(\sup_{0\leq s\leq A}|X^{\kappa}(s)|\geq (l_i^{\kappa}-2)\sqrt{\kappa}A/2\Big) \leq
\begin{cases}
2\exp\bigg\{-\frac14\frac{(l_i^{\kappa}-2)^2A\sqrt{\kappa A}}
{(l_i^{\kappa}-2)\sqrt{A} + 6\sqrt{\kappa A}}\bigg\}, &l_i^{\kappa} \geq 3 \\
1, &l_i^{\kappa}\leq 2.
\end{cases}
\end{equation}
An iteration of the estimates in (\ref{eq:Markovapp}) yields that the 
left-hand side of (\ref{eq:labelest}) is bounded from above by
\begin{equation}
\label{eq:plugin}
\sum_{\substack{(l_i^{\kappa})_{i=1}^{t/A}\colon \\ \sum l_i^{\kappa} = k_*(\kappa)}}
\prod_{\substack{i=1\\ l_i^{\kappa}\geq 3}}^{t/A} 
2\exp\bigg\{-\frac14\frac{(l_i^{\kappa}-2)^2A\sqrt{\kappa A}}
{(l_i^{\kappa}-2)\sqrt{A} + 6\sqrt{\kappa A}}\bigg\}.
\end{equation}
We have
\begin{equation}
\label{eq:sumlowerbound}
\sum_{\substack{i=1\colon \\ l_i^{\kappa}\geq 3}}^{t/A} \frac{(l_i^{\kappa}-2)^2}
{(l_i^{\kappa}-2)\sqrt{A}+6\sqrt{\kappa A}}
\geq \sum_{\substack{i=1\colon \\l_i^{\kappa}\geq 3}}^{t/A}
\frac{(l_i^{\kappa}-2)}{\sqrt{A} + 6\sqrt{\kappa A}}.
\end{equation}
Moreover,
\begin{equation}
\label{eq:lest}
\sum_{\substack{i=1\colon \\ l_i^{\kappa}\geq 3}}^{t/A} (l_i^{\kappa}-2) 
\geq k_*(\kappa)-\frac{4t}{A}.
\end{equation}
Inserting 
(\ref{eq:sumlowerbound}--\ref{eq:lest}) into (\ref{eq:plugin}), noting that \cite{HR18,E42}
\begin{equation}
\exists\,a,b>0\colon\,\mbox{ number of summands in right-hand side of (\ref{eq:labelest}) }
\leq a e^{b\sqrt{k_*(\kappa)}}/k_*(\kappa),
\end{equation}
and choosing $C_7$ and $\kappa$ large enough ($C_7 >4$ is sufficient), we get that the 
right-hand side of (\ref{eq:labelest}) is at most $e^{-C_8Ak_*(\kappa)}$. Here, $C_8$ is such 
that for all $k_*(\kappa)\geq C_7/A$ the inequality $(k_*(\kappa)-4t/A)/(1/\sqrt{\kappa} + 6)
\geq C_8k_*(\kappa)$ holds, which for $\kappa \geq 1$ is fulfilled when $C_7(1-7C_8)\geq 4$.
This yields the claim in \eqref{eq:controlonkkappa}.
\end{proof}


\section{Proof of Lemma~\ref{lem:eigenvalue}}
\label{Appendix*}

In Sections~\ref{B1}--\ref{B3} we prove a lemma that was used in Section~\ref{S8}. The 
proof is inspired by \cite[Theorem 12]{A10}.


\subsection{Neumann boundary conditions}
\label{B1}

In this section we recall the definition and some properties of the discrete Laplacian with 
Neumann boundary conditions. For further details we refer the reader to \cite{K07}.

Fix $x\in\Z^d$, $A>1$ and define the matrix $M_{B(x)}$ as
\begin{equation}
\label{eq:A}
M_{B(x)}(y,z)
= \begin{cases}
1, \mbox{ if } y,z\in B(x), \|y-z\|=1, \\
0, \mbox{ otherwise},
\end{cases}
\end{equation}
and the number of neighbors of $y$ in $B(x)$ as
\begin{equation}
\label{eq:n}
n_{B(x)}(y) = |\{z \in B(x)\colon\, \|y-z\|=1\}|.
\end{equation}

\begin{definition}
\label{def:neumann}
The Neumann Laplacian $\Delta_{B(x)}$ on $B(x)$ is defined via the formula
\begin{equation}
\label{eq:neumann}
\Delta_{B(x)} = M_{B(x)} - n_{B(x)},
\end{equation}
where $n_{B(x)}$ is the multiplication operator with the function $n_{B(x)}$.
\end{definition}

\begin{remark}
\label{rm:neumann}
The quadratic form associated with $\Delta_{B(x)}$ is given by
\begin{equation}
\label{eq:neumannform}
\Big\langle \Delta_{B(x)} f,g\Big\rangle 
=-\frac12 \sum_{\substack{y,z\in B(x)\\ \|y-z\|=1}}
[f(y)-f(z)]\,[g(y)-g(z)], \qquad f,g\in\R^{{B(x)}}.
\end{equation}
($\Delta_{B(x)}$ does not see that $B(x)$ is imbedded in $\Z^d$, which is why it is 
sometimes referred to as the graph Laplacian on $B(x)$.)
\end{remark}

\begin{lemma}
\label{lem:neumannproperties}
The following properties hold for all $x\in\Z^d$ and $A>1$.\\
(a) $\langle\Delta_{B(x)} f,f\rangle \leq 0$ for all $f\in\R^{B(x)}$.\\
(b) $\Delta_{B(x)}$ is self-adjoint.\\
(c) $\mathrm{ker}(\Delta_{B(x)}) = \R\one$, where $\one$ is the vector in $\R^{B(x)}$ 
with all entries equal to one.\\
(d) For all $f\in \ell^{2}(\Z^d)$,
\begin{equation}
\label{eq:laplacebound}
\langle\Delta f,f\rangle \leq \sum_{x\in \Z^d}
\Big\langle \Delta_{B(x)}f^{B(x)},f^{B(x)}\Big\rangle,
\end{equation}
where $f^{B(x)}$ is the restriction of $f$ to $B(x)$.
\end{lemma}

\begin{proof} 
Fix $x\in\Z^d$ and $A>1$.\\
(a) and (b) are consequences of Remark \ref{rm:neumann}.\\
(c) From Remark~\ref{rm:neumann} it is clear that constant functions are in the kernel of 
$\Delta_{B(x)}$. For the reverse direction, let $f\in \mathrm{ker}(\Delta_{B(x)})$. Again by 
Remark \ref{rm:neumann},
\begin{equation}
\label{eq:kernel}
0= \Big\langle \Delta_{B(x)}f,f\Big\rangle 
=  -\frac12 \sum_{\substack{y,z\in B(x) \\ \|y-z\|=1}} [f(y)-f(z)]^2.
\end{equation}
Hence, for all $y\in B(x)$ we have that $f(y)=f(z)$ for all $z$ such that $\|y-z\|=1$, $z\in B(x)$.\\
(d) Let $f\in\ell^{2}(\Z^d)$. Then
\begin{equation}
\label{eq:laplacerewrite}
\begin{aligned}
-2\langle \Delta f,f\rangle 
&= \sum_{\substack{y,z\in\Z^d \\ \|y-z\|=1}} [f(y)-f(z)]^2
= \sum_{x\in \Z^d} \sum_{y\in B(x)}\,\,\sum_{\substack{z\in\Z^d \\ \|y-z\|=1}} [f(y)-f(z)]^2\\
&\geq \sum_{x\in\Z^d}\sum_{y\in B(x)} \sum_{\substack{z\in B(x) \\ ||y-z||=1}} [f(y)-f(z)]^2
= -2\sum_{x\in \Z^d} \Big\langle \Delta_{B(x)} f^{B(x)}, f^{B(x)}\Big\rangle,
\end{aligned}
\end{equation}
where the second equality uses that $(B(x))_{x\in \Z^d}$ is a partition of $\Z^d$, while the third 
equality follows from Remark \ref{rm:neumann}. 
\end{proof}


\subsection{Proof of Lemma \ref{lem:eigenvalue} subject to a further lemma}
\label{B2}

Let $\|\cdot\|_2$ stand for both the Euclidean norm on $\R^{B}$ and the $\ell^2$-norm 
on $\ell^2(\Z^d)$.

\begin{lemma}
\label{lem:finitebox}
Subject to \eqref{eq:V}, there is a $\kappa_0>0$ such that, for all $\kappa\geq\kappa_0$, 
all $f\in\R^{B(x)}$ and all $x\in\Z^d$,
\begin{equation}
\label{eq:finitebox}
\Big\langle (\Delta_{B(x)}+\frac{1}{\kappa} V) f,f\Big\rangle 
\leq 4\frac{1}{\kappa}\delta \|f\|_2^2.
\end{equation}
\end{lemma}

\noindent
The proof of Lemma~\ref{lem:finitebox} is deferred to Section~\ref{B3}. First we complete
the proof of Lemma~\ref{lem:eigenvalue} subject to Lemma~\ref{lem:finitebox}.

\begin{proof}
Let $f\in\ell^{2}(\Z^d)$ and $\kappa\geq\kappa_0$, where $\kappa_0$ is chosen according to 
Lemma~\ref{lem:finitebox}. Then, by Lemma \ref{lem:neumannproperties}(d) and the fact that 
$(B(x))_{x\in\Z^d}$ is a partition of $\Z^d$, we may estimate
\begin{equation}
\label{eq:neumannapplication}
\Big\langle (\Delta +\frac{1}{\kappa} V) f,f\Big\rangle
\leq \sum_{x\in \Z^d}
\Big\langle \Big(\Delta_{B(x)}+\frac{1}{\kappa} V\Big) f^{B(x)}, f^{B(x)}\Big\rangle.
\end{equation}
Since $f=\sum_{x\in \Z^d} f^{B(x)}$, we have $\|f\|^2 = \sum_{x\in \Z^d} \|f^{B(x)}\|_2^2$. 
Combining (\ref{eq:finitebox}--\ref{eq:neumannapplication}), we get the claim.
\end{proof}


\subsection{Proof of Lemma \ref{lem:finitebox}}
\label{B3}

\begin{proof}
Fix $x\in\Z^d$ and $A>1$. First recall that, by Lemma \ref{lem:neumannproperties}(c), 
$\mathrm{ker}(\Delta_{B(x)}) = \R\one$, so that we can write each $f\in \R^{B(x)}$ 
as $f=\alpha \one + g$, $\alpha\in\R$, $g\in \big(\R\one\big)^{\perp}$. Therefore, 
using that $\Delta_{B(x)}$ is self-adjoint, see Lemma \ref{lem:neumannproperties}(b) and hence symmetric, we obtain that
\begin{equation}
\label{eq:neumannvalue}
\Big\langle \Big(\Delta_{B(x)}+\frac{1}{\kappa} V\Big)f,f\Big\rangle
=\Big\langle \Delta_{B(x)} g,g\Big\rangle
+ \frac{1}{\kappa} \big[\alpha^2 \langle V\one,\one\rangle 
+ 2\alpha\langle  V\one,g\rangle + \langle Vg,g\rangle\big].
\end{equation}
Using that the unit sphere intersected with $\big(\R\one)^{\perp}$ is compact, that 
$\Delta_{B(x)}$ is negative on $\big(\R\one\big)^{\perp}$, see Lemma \ref{lem:neumannproperties}(a), and that the scalar product
is continuous, we deduce that there is an $\eta>0$ such that $\langle \Delta_{B(x)} h,
h\rangle \leq -\eta$ for all $h\in\big(\R\one\big)^{\perp}$ with $\|h\|_2=1$. Hence
we may estimate the right-hand side of (\ref{eq:neumannvalue}) from above by
\begin{equation}
\label{eq:furtherestimate}
-\eta \|g\|_2^2 + \frac{1}{\kappa}\big[\alpha^2 \langle V\one,\one\rangle 
+ 2\alpha\langle  V\one,g\rangle + \langle Vg,g\rangle\big].
\end{equation}
Next, by \eqref{eq:V}, we have $\langle V\one,\one\rangle = \sum_{y\in B(x)} V(y) \leq 2\delta |B(x)|$.
An additional application of the Cauchy-Schwarz inequality shows that (\ref{eq:furtherestimate}) 
is at most
\begin{equation}
\label{eq:applicationofassumption}
-\eta \|g\|_2^2 +\frac{1}{\kappa}\big[\alpha^2 2\delta |B(x)| +2\alpha \|V\one\|_2\|g\|_2
+ \|Vg\|_2 \|g\|_2\big].
\end{equation}
Using the bound $\|Vg\|_2\leq \|V\|_{\infty} \|g\|_2$ (which also holds for $g$ replaced by $\one$) and the assumption that $V$ is bounded, we 
may further estimate (\ref{eq:applicationofassumption}) from above by
\begin{equation}
\label{eq:Vsmallerone}
-\eta \|g\|_2^2 + \frac{1}{\kappa}\big[\alpha^2 2\delta |B(x)|
+ 2\alpha \|V\|_{\infty} \|\one\|_2\|g\|_2 + \|V\|_{\infty}\|g\|_2^2\big].
\end{equation}
For any $a,b\in\R$ and $\gamma >0$ we have the inequality $2ab\leq \gamma a^2 + b^2/\gamma$. 
Pick $a=\alpha \|V\|_{\infty} \|\one\|_2$ and $b=\|g\|_2$. Then we may further estimate (\ref{eq:Vsmallerone}) 
from above by
\begin{equation}
\label{eq:appofbinominequ}
\begin{aligned}
& -\eta \|g\|_2^2 + \frac{1}{\kappa}\big[\alpha^2 2\delta |B(x)| 
+ \gamma \alpha^2 \|V\|_{\infty}^2 \|\one\|_2^2 + \|g\|_2^2/\gamma + \|g\|_2^2\big]\\
&= \|g\|_2^2\Big[-\eta +\frac{1}{\kappa}\Big(1+\frac{1}{\gamma}\Big)\Big]
+\frac{1}{\kappa}\big[2\delta+\gamma \|V\|_{\infty}^2\big]\alpha^2 \|\one\|_2^2,
\end{aligned}
\end{equation}
where we use that $\|\one\|_2^2=|B(x)|$. Now pick $\gamma =2\delta/\|V\|_{\infty}^2$ and note 
that, for $\kappa$ large enough so that $\frac{1}{\kappa}(1+1/\gamma-4\delta) \leq \eta$, we have 
$-\eta +\frac{1}{\kappa}(1+1/\gamma) \leq 4\frac{1}{\kappa}\delta$. Therefore, for $\kappa$ large 
enough we may estimate the right-most term in (\ref{eq:appofbinominequ}) from above by
\begin{equation}
\label{eq:finalest}
4\frac{1}{\kappa}\delta\Big[\|g\|_2^2+ \alpha^2 \|\one\|_2^2\Big] = 4\frac{1}{\kappa}\delta \|f\|_2^2.
\end{equation}
\end{proof}


\end{document}